\documentclass[11pt]{amsart}
\usepackage[letterpaper, margin=1in]{geometry}

\usepackage{hyperref}
\usepackage{cleveref}
\usepackage[disable,colorinlistoftodos,bordercolor=orange,backgroundcolor=orange!20,linecolor=orange,textsize=scriptsize]{todonotes}

\usepackage{graphicx, url,enumerate}
\usepackage{enumitem}
\usepackage{amssymb,mathtools}
 \renewcommand{\geq}{\geqslant}
 \renewcommand{\leq}{\leqslant}
 \renewcommand{\ge}{\geqslant}
 \renewcommand{\le}{\leqslant}

\usepackage{algorithm}
\usepackage{algpseudocode}
\usepackage{color}
\usepackage{tikz}
\usetikzlibrary{3d,calc}
\usepackage{eulervm} 

\numberwithin{equation}{section}
\numberwithin{algorithm}{section}

\theoremstyle{plain}
\newtheorem{theorem}{Theorem}[section]
\newtheorem{proposition}[theorem]{Proposition}
\newtheorem{lemma}[theorem]{Lemma}
\newtheorem{corollary}[theorem]{Corollary}

\newtheorem{definition}[theorem]{Definition}

\theoremstyle{remark}
\newtheorem{remark}[theorem]{Remark}

\newcommand{\Deltap}{\Delta_{+}}
\newcommand{\Deltapp}{\Delta_{++}}
\newcommand{\Rpp}{\R_{++}}
\newcommand{\Qpp}{\Q_{++}}

\newcommand{\Rp}{\R_{+}}

\newcommand{\Zp}{\Z_{+}}

\newcommand{\bard}{\bar{d}}

\newcommand{\fmin}{f_{\min}}

\def\<#1>{\langle #1\rangle}

\DeclareMathOperator{\E}{E}

\DeclareMathOperator{\diag}{diag}

\DeclareMathOperator{\sgn}{sgn}

\DeclareMathOperator{\supp}{supp}

\newcommand{\fO}{\mathfrak{O}}
\newcommand{\N}{\mathbb{N}}
\newcommand{\C}{\mathbb{C}}
\newcommand{\Q}{\mathbb{Q}}
\newcommand{\R}{\mathbb{R}}
\newcommand{\Z}{\mathbb{Z}}
\newcommand{\s}{\mathbf{s}}
\newcommand{\g}{\mathbf{g}}
\newcommand{\uu}{\mathbf{u}}
\newcommand{\x}{\mathbf{x}}

\newcommand{\y}{\mathbf{y}}
\newcommand{\z}{\mathbf{z}}

\newcommand{\w}{\mathbf{w}}

\newcommand{\cA}{\mathcal{A}}
\newcommand{\cB}{\mathcal{B}}
\newcommand{\cC}{\mathcal{C}}

\newcommand{\cE}{\mathcal{E}}
\newcommand{\cF}{\mathcal{F}}
\newcommand{\cG}{\mathcal{G}}
\newcommand{\cH}{\mathcal{H}}
\newcommand{\cI}{\mathcal{I}}

\newcommand{\cM}{\mathcal{M}}

\newcommand{\cN}{\mathcal{N}}

\newcommand{\cT}{\mathcal{T}}

\newcommand{\cW}{\mathcal{W}}

\newcommand{\ba}{\mathbf{a}}
\newcommand{\bA}{\mathbf{A}}
\newcommand{\bb}{\mathbf{b}}
\newcommand{\bc}{\mathbf{c}}

\newcommand{\be}{\mathbf{e}}

\newcommand{\bF}{\mathbf{F}}

\newcommand{\bG}{\mathbf{G}}

\newcommand{\zero}{\mathbf{0}}

\newcommand{\br}{\mathbf{r}}
\newcommand{\bs}{\mathbf{s}}

\newcommand{\bu}{\mathbf{u}}
\newcommand{\bU}{\mathbf{U}}
\newcommand{\bV}{\mathbf{V}}
\newcommand{\bv}{\mathbf{v}}
\newcommand{\bw}{\mathbf{w}}
\newcommand{\bx}{\mathbf{x}}
\newcommand{\by}{\mathbf{y}}
\newcommand{\bz}{\mathbf{z}}
\newcommand{\rB}{\mathrm{B}}
\newcommand{\rC}{\mathrm{C}}
\newcommand{\rD}{\mathrm{D}}

\newcommand{\rH}{\mathrm{H}}

\newcommand{\rK}{\mathrm{K}}
\newcommand{\rL}{\mathrm{L}}
\newcommand{\rP}{\mathrm{P}}

\newcommand{\rS}{\mathrm{S}}

\newcommand{\0}{\mathbf{0}}
\newcommand{\1}{\mathbf{1}}

\newcommand{\dist}{\mathrm{dist}}

\begin{document}
\title[Complexity of Geometric Programming in the Turing Model]{Complexity of Geometric Programming in the Turing Model and application to Nonnegative Tensors}
\author{Shmuel  Friedland
  and\;
 St\'ephane Gaubert
 }
\date{March 16,  2025
}
\address{
 Department of Mathematics, Statistics and Computer Science,
 University of Illinois at Chicago, Chicago, Illinois 60607-7045,
 USA, \texttt{friedlan@uic.edu},
 }
\address{
INRIA and Centre de Math\'ematiques Appliqu\'ees (CMAP), \'Ecole polytechnique, IP Paris, UMR 7641 CNRS, 91128 Palaiseau C\'edex, \texttt{stephane.gaubert@inria.fr} }
\subjclass[2010]{
05C50,05C65,15A69,68W25,90C25}

\keywords{
Nonnegative tensors, spectral radius,  minimum of convex functions, spectral radius, clique number of hypergraphs, interior-point methods
}
 
\begin{abstract}\bf{
We consider a version of geometric programming problem consisting in
minimizing a function given by the maximum of finitely many log-Laplace transforms of discrete nonnegative measures on a Euclidean space.  
  Under a coerciveness assumption, we show that an $\varepsilon$-minimizer
  can be computed in a time that is polynomial in the input size and in $|\log\varepsilon|$. This is obtained by establishing bit-size estimates on approximate
  minimizers and by applying the ellipsoid method. We also derive
  polynomial iteration complexity bounds for the interior-point method
applied to the same class of problems.
We deduce that the spectral radius of a partially symmetric, weakly irreducible nonnegative tensor can be approximated within an $\varepsilon$-error in polynomial time.  
For strongly irreducible tensors, we show in addition that the logarithm of the positive eigenvector is polynomial time approximable.
Our results also yield that the the maximum of a nonnegative homogeneous $d$-form in the $\ell_d$ unit ball can be approximated in polynomial time.   In particular, the spectral radius of uniform weighted hypergraphs and some known upper bounds for the clique number of uniform hypergraphs are polynomial time computable. In contrast, we provide an example
showing that the Phase I approach needs exponentially many bits
to solve the feasibility problem in geometric programming.}
\end{abstract}
\maketitle
\tableofcontents
\section{Introduction}
\label{sec:sum}
\subsection{Statement of the main results}\label{subsec:majrt}
The log-Laplace transform of a nonnegative measure $\mu$ on $\R^n$ is the map
$\R^n\to \R\cup\{+\infty\}$ defined by $h(\x) = \log \int \exp(\y^\top \x)d \mu(\y)$. (Here and throughout, we consider the elements of $\R^n$ as column vectors.)
Maps of this form are convex~\cite{Kin61}.
They have been widely studied in large deviations theory and in convex geometry, see e.g.~\cite{brazitikos}. 
Focusing on computational aspects, we consider measures with finite support. In this case, a log-Laplace transform can be expressed as 
\begin{align}
  \label{e-llt} h(\x)=\log\sum_{\ba\in\cA} h_{\ba} \exp(\ba^\top \x)\enspace.
\end{align}
where $\cA=\{\ba_1,\ldots,\ba_N\}$ is a nonempty finite subset of $\R^n$, and
$\{h_{\ba}, \ba\in \cA\}$ is a family of positive scalars.
Log-Laplace transforms of the form~\eqref{e-llt} will be referred to as {\em llt functions}. Affine functions are a special case of llt functions, corresponding to the situation where $\cA$ is a singleton.

Assume that $f_0,\ldots,f_N$ are llt functions.
For a positive integer $N$ we let $[N]\coloneqq\{1,\ldots,N\}$,  and set $[0]\coloneqq\emptyset$.  The problem of {\em geometric programming} can be formulated as the following convex optimization problem:
\begin{equation}\label{geoprogst}
\begin{aligned}
&\textrm{minimize} \quad f_0,\\
  &\textrm{subject to} \quad f_i\le 0, \quad i\in[N],
\end{aligned}
\end{equation}
where $N$ is a nonnegative integer. Since affine functions are a special
case of llt functions, this formulation includes linear programming as a special case.

The case $N=0$ has received significant attention~\cite{BLNW20, SV19}.  
With an appropriate change of variable, this case can be extended to {\em signomial
   programming} problems. Their dual formulations are {\em entropy
  maximization} problems, which appear in various applications.
Notable examples include relaxations of polynomial optimization problems
 via sums of nonnegative circuits~\cite{CS16,Dressler2017,DresslerNaumannTheobald+2021+227+236,Theobald2023}, risk-sensitive control~\cite{AB17},
 and games of topological entropy~\cite{asarin,AGGG}.
 
In this paper we consider the following unconstrained version of the geometric programming problem. Specifically, let $f$ be the pointwise maximum of $N$  llt functions
\begin{equation}\label{deff(x)fj}
f(\x)=\max_{j\in[N]}f_j(\x),\qquad 
f_j(\x)=\log\sum_{\ba\in\cA_j} f_{\ba,j} \exp(\ba^\top \x), \enspace.
\end{equation}
Our focus is on approximating the value
\begin{align}
  \label{e-geo} f_{\min}=\inf_{\x\in\R^n} f(\x)
\end{align}
as well as computing an approximate minimizer.

This problem has two motivations.   The first one,  is the feasibility problem in geometrical programming \eqref{geoprogst}: Is the set $\{\x\in\R^n, f_i(\x)\le 0, i\in[N]\}$ nonempty?  See Subsection \ref{subsec:appl}. The second motivation stems from our study of the complexity in the Turing model (also known as the bit model)
of the computation of the spectral radius of partially symmetric tensors, see Subsection \ref{subsec:intros}.
We will assume here that the problem \eqref{e-geo} is well posed, in the sense that
the map $f$ is coercive, meaning that it has compact sublevel sets. This
ensures that the minimum of $f$ is achieved. We shall see that $f$
is coercive iff it satisfies the following elementary condition
\begin{equation}\label{coercons}
\{\x\in\R^n,\ba^\top \x\le 0 \textrm{ for } \ba\in \cA\}= \{\0\},  \quad\text{where } \cA:=\cup_{j\in[N]}\cA_j.
\end{equation}
This condition can be verified in polynomial time by reduction
to a linear programming problem.
Our first result shows that an approximate solution of Problem~\eqref{deff(x)fj}
can be obtained in polynomial time.  Following~\cite[p.~29]{GLS88}, we denote by $\<r>$ the number of bits needed to encode a rational number $r$, and if $\ba \in \Q^n$, we denote by $\<\ba>=\sum_{1\leq i\leq n} \<a_i>$ the encoding length of $\ba$.
The encoding length of $f_i$ is defined as $\<f_i>=\sum_{\ba \in \cA_i} (\<\ba> +\<f_{\ba,i}>)$.
We finally define the encoding length of $f$ as $\langle f\rangle:= \sum_{1\leq i\leq N} \<f_i>$.
\begin{theorem}[The Turing model for $f_{\min}$, see~\Cref{mainthm}]
  \label{mainthm-intro}
  For all $\varepsilon\in \Qpp$, one can compute a vector $\bx^\star\in \Q^n$
  such that $f(\bx^\star)\leq \fmin+ \varepsilon$, in a time that is polynomial
  in $\<f>$ and in $|\log\varepsilon|$.
\end{theorem}
We establish this theorem using the ellipsoid method~\cite{GLS81,GLS88}.
A crucial technical step in this approach is to bound approximate
minimizers. 
This is achieved by introducing the {\em coercivity constant}
with respect to a vector subspace $\mathbf{E}\subset \R^n$
\begin{equation}\label{def-coercive}
\nu(\cA,\mathbf{E})=\min_{\x\in \mathbf{E}, \; \|\x\|_\infty=1}\max\{\ba^\top \x,  \ba\in\cA\} \enspace .
\end{equation}
When $\mathbf{E}=\R^n$, we simplify the notation to $\nu(\cA):=\nu(\cA,\R^n)$.

The following result provides an explicit bound for the minimizers of $f$,
showing that the logarithms of the entries of the minimizer are polynomially
bounded in the input size.
\begin{theorem}[Bound on minimizers]\label{th-localize}
  Suppose that $f$ is coercive. Then, any optimal solution $\bx^\star$ satisfies
  \[
  \|\bx^\star\|_\infty \leq \frac{f(\zero)-t_{\min}(f)}{\nu(\cA)} ,\qquad \text{where }
  t_{\min}(f):=\min\{\log f_{\ba,j}, \ba\in \cA_j, j\in[N]\} \enspace.
  \]
Moreover, if $\cA_j \subset \Z^n$ holds for all $j\in [N]$, we have
  \begin{align}\label{boundnu}
  \nu(\bA) \ge \frac{1}{\prod_{i=1}^n \sqrt{\|\bb_i\|^2 +1}} \enspace,
  \end{align}
  in which the elements of $\cA=\{b_1, \dots,b_M\}$ are ordered by decreasing Euclidean norms, i.e.,
  $\|b_1\|\geq \dots \geq \|b_M\|$.
\end{theorem}
To establish \Cref{mainthm-intro}, 
we reformulate~\eqref{e-geo} as the problem consisting in minimizing the
linear form $(\bx,t)\mapsto t$ over the epigraph of $f$, $\operatorname{Epi}(f)=
\{(x,t)\in \R^n\times \R \mid t\geq f_j(x), 1\leq j\leq N\}$.
Using~\eqref{th-localize},
we show that the minimization problem can be restricted  to an explicit
convex subset $\rK\subset \R^n\times \R$ which is
sandwitched between balls with radii
polynomially bounded in the input size.
We also use the fact that the logarithm and exponential maps can be approximated
in polynomial time on restricted domains~\cite{borwein}, together with a renormalization argument,
to construct a polynomial time weak separation oracle for this set $\rK$,
leading finally to~\Cref{mainthm-intro}.

We now apply~\Cref{mainthm-intro} to compute the spectral radius of a
nonnegative homogeneous map. We consider a $(d-1)$-homogeneous
map $\bF=(F_1,\ldots,F_{n}):\Rpp^{n}\to\Rpp^{n}$,
such that 
\begin{align}
F_i (\bz) = \sum_{\ba\in \cA_i} f_{\ba,i} \bz^\ba, \quad \bz^\ba:= z_1^{a_1}\dots z_n^{a_n},
\qquad \cA_i \subset \N^n \cap (d-1) \Deltap^n, \; f_{\ba,i}\in \Qpp
\text{ for } \ba \in \cA_i \enspace .\label{e-sparse}
\end{align}
Here, $\Deltap^n=\{\bx\in\Rp^n\mid \sum_{i} x_i = 1\}$ denotes the simplex.
The condition $\cA_i \subset (d-1) \Deltap^n$ ensures
that $F(\lambda \bz)= \lambda^{d-1} F(\bz)$ for all $\lambda>0$
and $\bz\in \Rpp^n$.
The input size is $\<F> = \sum_{i\in[n]}\sum_{\ba\in \cA_i} (\<\ba>+\<f_{\ba,i}>)$.
As detailed in~\Cref{subsec:intros}, such maps are associated to partially symmetric nonnegative tensors.

A vector $\z\in\C^n\setminus\{\0\}$ is called an eigenvector of $\bF$ with an eigenvalue $\lambda$ if $\bF(\z)=\lambda\z^{\circ(d-1)}$, where $\z^{\circ(d-1)}=(z_1^{d-1},\ldots,z_n^{d-1})^\top$.  
The maximal value of $|\lambda|$, as $\lambda$ ranges over the eigenvalues,
is called the {\em spectral radius} of $\bF$ and is denoted by $\rho(\bF)$.
We say that $\bF$ is {\em weakly irreducible} if for  $\bz=\1:=(1,\ldots,1)^\top$ the Jacobian $D(\bF)(\1)=[\frac{\partial F_i}{\partial x_j}](\1)$ is an irreducible  matrix.  
If $\bF$ is weakly irreducible, then there is an eigenvector $\z\in\Rpp^n$, unique up to a scalar factor, and the associated eigenvalue $\lambda$ coincides
with $\rho(\bF)$ \cite{FGH13}. Furthermore,  $\rho(\bF)$ has the following ``Collatz-Wielandt'' minimax and maximin characterization \cite[(3.13)]{FG20}:
\begin{equation}\label{minmaxchar}
\rho(\bF)=\min_{\z\in \Rpp^n} \max_{i\in[n]} \frac{F_i(\z)}{z_i^{d-1}}=\max_{\z\in\Rp^n\setminus\{\0\}}\min_{i\in[n],z_i>0} \frac{F_i(\z)}{z_i^{d-1}} \enspace .
\end{equation}
The equality holds if and only if $\z$ is proportional to $\bu$.
\begin{theorem}\label{Fcoercthmnew} Let $\bF=(F_1,\ldots,F_{n}):\Rpp^{n}\to\Rpp^{n}$ be a weakly irreducible $(d-1)$-homogeneous map, and let
  \begin{align}\label{e-def-cw}
  f(\bx) := \max_{i\in [n]} (\log F_i(\exp(\bx))-(d-1)x_i )
  \end{align}
  Then, an $\varepsilon$ minimizer $\bx^\star$ of $f$ can be computed in a time that
  is polynomial in $\<F>$ and in $|\log\varepsilon|$, and $f(\bx^\star)$ provides
  an $\varepsilon$-approximation of $\log \rho(\bF)$.
\end{theorem}
A key technical element in the proof is the following. As before, we set 
$\cA:= \cup_{i\in[n]}\cA_i$.
\begin{lemma}
  Suppose that $\bF$ is weakly irreducible.
  Then, the restriction of the map $f$ in~\eqref{e-def-cw}
  to the hyperplane $\mathbf{E}_n:= \{\bx \in \R^n\mid x_n=0\}$
  is coercive, moreover,
\begin{align}
\nu(\cA,\mathbf{E}_n)\ge \big(4(d-1)^2+1\big)^{-(n-1)/2}\enspace. 
\end{align}
\end{lemma}
The space $\mathbf{E}_n$ is introduced for technical convenience.
More generally, any hyperplane that is transverse to the line generated by the unit vector and possesses a concise rational encoding would yield a useful
bound.
Then, \Cref{Fcoercthmnew} is deduced from~\Cref{mainthm-intro}.

Using \Cref{Fcoercthmnew}, and using an a priori bound on the spectral radius, we establish the following result:
\begin{corollary}\label{cor-specrad}
  The spectral radius of a weakly-irreducible $(d-1)$-homogeneous
  map $\Rp^n\to \Rp^n$ can be approximated
  in polynomial time. 
\end{corollary}
Exploiting the relation between the tensor eigenproblem and the maximization
of a homogeneous form $g$ over the a $\ell_d$ ball (see~\eqref{geignvalvecintro}),
we obtain the following additional corollary.
\begin{corollary}\label{poltcompmudgintro}  Let $g$ be a $d$-homogeneous form with nonnegative rational coefficients, for $d\ge 2$.  Let $\mu_d(g) = \max\{ g(z)\mid \|z\|_d\leq  1\}$.
Then an $\varepsilon$-approximation of $\mu_d(g)$ can be computed in polynomial time in $\varepsilon$ and the coefficients of $g$.
\end{corollary}
This is in contrast to a classical result of
Motzkin and Straus~\cite{MS65}, showing that $A$ is the adjacency matrix of a simple undirected graph $G$ on $n$ vertices with clique size $\omega(G)$, then
\begin{equation}\label{MSeq}
\mu_1(\z^\top A\z)=1-\frac{1}{\omega(G)}.
\end{equation}
In particular, it is NP-hard to approximate $\mu_1(\z^\top A\z)$.

\Cref{poltcompmudgintro} applies to higher-order Markov chains \cite{NQZ09},  first eigenvalue of nonnegative symmetric tensors \cite{CDN14} and to the spectral radius of $d$-uniform weighted hypergraphs \cite{SS82,RBP09,XQ15}.

Fixing a coordinate $i\in[n]$, we call {\em normalized eigenvector} of a weakly irreducible map $\bF$ the only positive eigenvector $\bu$ such that $u_i=1$.
\begin{theorem}\label{th-cex}
  There is no polynomial time algorithm taking as input
a weakly-irreducible $(d-1)$-homogeneous
  map $\Rp^n\to \Rp^n$ 
and returning an $O(1)$ approximation of a normalized eigenvector.
\end{theorem}
To show this result, we rely on the following counter-example, in which the normalized
eigenvectors has entries which are doubly exponential in the input size.
Let $W,d$ be integers with $W\geq 2$ and $d\geq 3$.
We define the homogeneous map of degree $d-1$, $\bF: \Rp^{2n}\to \Rp^{2n}$, 
\begin{equation}\label{Fexamp}
\begin{aligned}
  &F_i(\z)=W z_1^{d-2} z_{i+1} \textrm{ for } i=1,\ldots,n,\\
  &F_{n+i}(\z)=W^{-1} z_{n+1}^{d-2} z_{n+1+i} \textrm{ for } i=1,\ldots,n-1,\\
&   F_{2n}(\bz) =W^{-1} z_1z_{n+1}^{d-2} \enspace .
\end{aligned}
\end{equation}
We shall see that the map $\bF$ is weakly irreducible, so it has an eigenvector
$\bu\in \Rpp^{2n}$ which is unique up to a multiplicative constant. 
We will show that 
$u_{n+1}/u_1 \leq W^{-(d-1)^n}$, which
entails that every $O(1)$ approximation of the eigenvector of $\bF$ normalized by $u_{n+1}=1$
must use an exponential number of bits, leading to~\Cref{th-cex}.
We explain the intuitive probabilistic interpretation behind
this counter-example in~\Cref{rk-proba}.

However, the next result shows that the {\em logarithm} of the eigenvector
can be approximated in polynomial time. To do so, we make a stronger technical
assumption, requiring that $\bF$ be {\em strongly irreducible}, meaning that
the differential $D(\bF)$, evaluated at the unit vector $\1$, is a matrix
with positive off-diagonal entries
\begin{theorem}\label{sirFthm}
  Assume that $\bF$ is strongly irreducible.   Then $\nu(\cA,\mathbf{E}_n)\ge 1$.
  Moreover, for all $\varepsilon\in \Qpp$, an $\varepsilon$-approximation
  of the log of a positive eigenvector of $\bF$ can be computed in a time
  which is polynomial in $\<\bF>$ and $|\log\varepsilon|$.
\end{theorem}
The previous results are established by the ellipsoid method, which is
mostly a theoretical tool. We discuss in~\Cref{sec:intmethod} the application of
the more efficient interior-point methods to the geometric programming problem~\eqref{e-geo}.

We also consider {\em posynomial} maps. The latter extend the maps of the form~\eqref{e-sparse}, allowing the exponents to take rational values. 
We shall see in subsection  \ref{subsec:quashommapspol} 
that Theorems~\ref{Fcoercthmnew} and~\ref{sirFthm}
carry over to homogeneous posynomial maps.
We shall also see in subsection  \ref{subsec:pcmupg} that if $g$ is a homogeneous posynomial form with rational degree $d \ge 2$, and for every rational number $p\ge d$,
the maximum of $g$ over the $\ell_p$ ball,  $\mu_p(g)$,
can be approximated in polynomial time.
\subsection{Motivation from Nonnegative Tensors}\label{subsec:intros}
Multidimensional arrays with $d\ge 3$ indices, which are natural generalization of matrices $d=2$, are ubiquitous  in data science \cite{Liu22},  machine learning \cite{JWLL19},  DNA analysis \cite{OGA07}, mathematics \cite{Lan12}, numerical analysis \cite{Lim21}, optimal transport \cite{Fri20}, quantum information \cite{FECZ22}, as well as other fields.  Many computational problems, concerning tensors, are NP-hard~\cite{HL13}.

A $d$-tensor is an element of the $d$-fold tensor product $\R^{n^{\times d}}:=\otimes^d\R^n$; it is a $d$-dimensional array
with real elements, $\cT=[t_{i_1,\ldots,i_d}], i_1,\ldots,i_d\in[n]:=\{1,\ldots,n\}$.  It is symmetric if $t_{i_{\sigma(1)},\ldots,i_{\sigma(d)}}=t_{i_1,\ldots,i_d}$ for all bijections $\sigma:[d]\to[d]$.  The subspace of symmetric tensors, denoted as $\rS^d\R^n$,  can be identified
to the space $d$-homogeneous polynomials in $n$ variables $\rP(n,d)$.
Indeed, to a symmetric tensor corresponds a real multivariate
polynomial sending $\z=(z_1,\ldots z_n)^\top$ to 
$p(\z)= \sum_{i_1=\cdots=i_d=1}^n t_{i_1,\cdots,i_d}z_{i_1}\cdots z_{i_d}$, and this
correspondence is bijective.   A tensor $\cT$ is called partially symmetric
(in the last $d-1$ variables) if $t_{i_1,i_{\sigma(2)}\ldots,i_{\sigma(d)}}=t_{i_1,\ldots,i_d}$ for all bijections $\sigma: \{2,\dots,d\}\to \{2,\dots,d\}$.
The space of partially symmetric tensors is denoted by $\R^{n^{\times d}}_{ps}$.
A partially symmetric tensor $\cF=[f_{i_1,\ldots,i_d}]\in\R^{n^{\times d}}_{ps}$
corresponds a homogeneous map of degree $d-1$:
\begin{equation}
\bF(\z)=(F_1,\ldots,F_n):\R^n\to \R^n, \qquad 
F_i=\sum_{i_2,\ldots,i_d=1}f_{i,i_2,\ldots,i_d}z_{i_2}\cdots z_{i_d},
\label{e-hommap}
\end{equation}
and again this correspondence is bijective.

Our motivation stems from partially symmetric tensors with {\em nonnegative} elements, or simply {\em nonnegative tensors}. The set of these tensors is denoted as $\R^{n^{\times d}}_{ps,+}$.
A tensor $\cF\in \R^{n^{\times d}}_{ps,+}$ defines a map $\bF$, 
as in~\eqref{e-hommap}, sending $\Rp^n$ to itself.
The tensor $\cF$ is called {\it weakly irreducible}
if the corresponding  map $\bF$ is. The eigenvalues
of $\cF$ are defined as the eigenvalues of this map; see~\cite{Lim05,Qi05,FGH13}.

The tensor eigenproblem arises in particular when maximizing,
over the $\ell_p$ ball,
a homogeneous form of degree $d$ with nonnegative coefficients:
\begin{equation}\label{maxprobqintro}
\mu_p(g):=\max_{\|\z\|_p\le 1} g(\z), \quad p\in [1,\infty].
\end{equation}
This  has been studied in~\cite{LS47,Lim05,Qi05}:
the solution is a positive vector $\w$ such that 
\begin{equation}\label{geignvalvecintro}
\begin{aligned}
  &\frac{1}{d}\nabla g(\w)=g(\w)\w^{\circ (p-1)},  \quad\|\w\|_p=1\enspace .
\end{aligned}
\end{equation}
When $p=d$, we recover the tensor eigenproblem.
More generally, the eigenproblem
approach allows one to handle any $p\geq d$ after appropriate transformations.

Using the minimax characterization~\eqref{minmaxchar},
we see that the computation of the logarithm
of the spectral radius of a nonnegative tensor corresponds to $f_{\min}$ in \eqref{e-geo}.

\subsection{Applications}\label{subsec:appl}

We now present three applications of our results.

\noindent
(a) Consider the {\em feasibility problem} in geometric
programming \eqref{geoprogst} :
\begin{align}\label{e-feasibility}
  \text{Decide whether there exists }\bx\in \R^n \text{ such that }
  f_i(\bx) \leq 0,\text{ for all }i\in [N] \enspace ,
\end{align}
where every $f_i$ is of the form~\eqref{deff(x)fj}.
The natural approach to decide feasibility is to use
the following ``Phase I'' problem, in the spirit of linear
programming:
\begin{align}
\tag{Phase I}\label{phi1}\qquad
\min t,  \qquad t \geq f_i(\bx), \text{ for all }i\in [N], \qquad (\bx,t)\in \R^n\times \R
\enspace.
\end{align}
Indeed, still assuming that $f=\sup_{i\in[N]}f_i$ is coercive, so that the value of the Phase I problem is attained, one gets that the answer to the feasibility Problem~\eqref{e-feasibility} is affirmative if, and only if, the value $t^\star$ of the Phase I problem satisfies $t^\star\leq 0$. The following result shows that the Phase I approach fails
to decide feasibility in polynomial time.
\begin{theorem}[Phase I needs exponentially many bits]\label{th-phase1}
  There is a family of instances of the feasibility problem in geometric programming,
  for coercive functions $f$, containing infinitely many feasible and unfeasible instances, such that
  for every instance,
  the value $t^\star$ of the Phase I problem satisfies $t^\star= 2^{2^{-\Omega(\<f>)}}$, the multiplicative constant in the $\Omega(\cdot)$ term being independent of the instance.
\end{theorem}
This family of instances is constructed by considering a perturbation
of the map~\eqref{Fexamp}.\\

\noindent
(b) In a number of applications, one has to solve a minimization problem of the form
\begin{align}
  \min\{\bc^\top x: \ f_i(\x)\le 0, i\in[N]\} \enspace,\label{e-optimfeas}
\end{align}
where every $f_i$ is of the form~\eqref{deff(x)fj} and $f$ is coercive,
see~\cite[Section 5.1]{Dressler2017}. \Cref{th-phase1} entails that there is no hope to get a universal polynomial-time approximation result by off-the shelve methods like the ellipsoid (or the interior points). However, the difficulty
of the feasibility problem suggests to introduce a condition number
$\kappa(f)=( \max(0,-f_{\min}))^{-1}$, so that $\kappa<\infty$ whenever
problem~\eqref{e-feasibility} is strictly feasible. Then, it
follows from the present approach that when $f$ is coercive,
the optimization problem~\eqref{e-optimfeas} restricted
to instances with an a priori upper bound on $\kappa$ can be solved
approximately in polynomial time. We leave the more substantial aspects of the complexity of the optimization problem~\eqref{e-optimfeas} as an open question.\\

\noindent
(c) The geometric programming problem providing the logarithm of the spectral radius
has a dual version, which involves
an entropic programming problem. Following~\cite[Section 6.3]{FG20},
a tensor $\mu=[\mu_{i_1,\ldots,i_d}]\in \R_{ps,+}^{n^{\times d}}$ is called
an {\em occupation measure} if
\begin{eqnarray*}
 \sum_{i_1\in [n],\dots, i_d\in [n]}\mu_{i_1,i_2,\dots, i_d}=1,
 \sum_{i_2,\dots,i_d\in [n]} 
 \mu_{j,i_2,\dots,i_d} =\displaystyle\sum_{i_1,i_3,\dots,i_d\in [n]
}
 \mu_{i_1,j,i_3,\dots,i_d} ,
\end{eqnarray*}
for all $j\in [n]$.  
\todo[inline]{SF: I think that this is the correct statement for occupational measure.  When one sums on $j$ in the second equality one gets the equality $1=1$!}
We denote by $\fO(n^{\times (d-1)})\subset \R_{ps,+}^{n^{\times d}}$ the set of occupation measures.
For $\cT\in\R_{ps,+}^{n^{\times d}}$ we denote by 
$\supp \cT$ the \emph{support} of the tensor $\cT$, i.e.,
$
\supp \cT:= \{(i_1,\dots,i_d)\mid t_{i_1,\dots, i_d}>0 \} $, and by
$\fO(n^{\times (d-1)},\supp \cT)\subseteq \fO(n^{\times (d-1)})$ the set of occupation measures whose support is contained in $\supp \cT$.

Assume that $\bF$ is weakly irreducible.
Then there exists a unique positive  eigenvector $\bu>\0$  (up to rescaling) such that
$\bF(\bu)=\rho(\bF)\bu^{\circ(d-1)}, \bu>\0$.
Let $D(\bF)(\x)=[\frac{\partial F_i(\x)}{\partial x_j}]$ be the Jacobian matrix of $\bF$.
Set $A:=\diag(\bu)^{-(d-2)}D(\bF)(\bu),  \diag(\bu)=\diag(u_1 ,\cdots,u_{n})$.
Recall that $A$ is a nonnegative irreducible matrix that satisfies \cite[(3.18)]{FG20}:
\begin{equation*}
A\bu=(d-1)\rho(\bF)\bu, \quad A^\top\w=(d-1)\rho(\bF)\w,\quad\w^\top \uu=1, \qquad \bw\in \Rpp^n \enspace .
\end{equation*}
Theorem 19 in \cite{FG20} states that
the spectral radius of  a nonzero nonnegative map $\bF=(F_1,\ldots,F_{n})^\top$,  has the following dual characterization 
\begin{equation}
\begin{aligned}
\log \rho(\bF) = \max_{\mu\in \fO(n^{\times(d-1)})}
\sum_{i_1,\dots,i_d\in [n]}
\mu_{i_1,\dots, i_d} \log \Big(\frac{(\sum_{k_2,\dots, k_d}\mu_{i_1,k_2,\dots, k_d})f_{i_1,\dots, i_d}}{\mu_{i_1,\dots, i_d}}\Big) \enspace .
\end{aligned}\label{e-dual}
\end{equation}
Moreover, the unique optimal solution of the dual problem is given by
\begin{equation*}
\mu^\star_{i_1,\ldots,i_d}=\frac{1}{\rho(\bF)} w_{i_1}u_{i_1}^{-(d-2)}f_{i_1,\ldots,i_d}u_{i_2}\cdots u_{i_d} \textrm{ for } i_1,\ldots,i_d\in[n].
\end{equation*}
We are interested in the computation of $\mu^\star$. Consistently with~\eqref{e-sparse},
we take into account sparsity and symmetry to encode the occupation measure  $\mu^\star$.
I.e., we write $\mu^\star=(\mu^\star_{i\cdot})_{i\in [n]}$, and for all $i\in[n]$, we identify $\mu^\star_{i\cdot}$ to a map $\cA_i \to \Rpp$. Then, we denote by $\log\mu^\star$ the entrywise logarithm
of $\mu^\star$.
As a corollary of~\Cref{sirFthm}, we get:
\begin{corollary}
  If $\bF$ is strongly irreducible, and if $\mu^\star$ denotes the optimal solution
  of the dual problem~\eqref{e-dual},
  then, we can approximate $\log\mu^\star$ in polynomial time.
  \end{corollary}
(c) Let $2^{[n]}_d$ denote all subsets of $[n]$ of cardinality $d\in[n]$.
Assume that $\cE\subset 2^{[n]}_d$.  A simple $d$-uniform hypergraph corresponding to the set of hyperedges $\cE$, denoted as $\cH=([n],\cE)$
can be described by a symmetric tensor $\cH(\cE)=[h_{i_1,\ldots,i_d}]\in \rS^d\R^{n}$, where
$h_{i_1,\ldots,i_d}=1$ if $\{i_1,\ldots,i_d\}\in\cE$, and $0$ otherwise.
A weighted $d$-uniform hypergraph $\cH=([n],\cE,\cW(\cE))$ is given by
$\cW(\cE)=[w_{i_1,\ldots,i_d}]\in \rS^d_+\R^{n}$, where $w_{i_1,\ldots,i_d}>0\iff \{i_1,\ldots,i_d\}\in\cE$.
Associate with $\cH$ the multilinear polynomial $g(\x)=d! \sum_{\{i_1,\ldots,i_d\}\in\cE}x_{i_1}\cdots x_{i_d}$.  Then $\mu_d(g)$ is the spectral radius of the uniform weighted hypergraph $\cH(\cE)$.  If $g$ is irreducible then eigenvector $\bu>0$ defined as in (a) is the unique eigenvector corresponding to $\cH(\cE)$.  
Our results imply that $\mu_d(g)$  can be approximated in polynomial time.  If $g$ is strongly irreducible then $\log\bu$ can be approximated in polynomial time. 

A clique $\cC$ in $\cE$ is subset of $[n]$ such that $2^{\cC}_d\subset \cE$.
Denote by $\omega(\cE)$ the maximum cardinality of a clique in $\cE$.  Then
$\prod_{i=1}^{d-1} (\omega(\cE)-i)\le \rho(\cH(\cE))$.  
This inequality is sharp if $\cE$ is a union of disjoint cliques in $2^{[n]}_d$.
As $\rho(\cH(\cE))$ is polynomial time approximable, the above inequality gives a computable upper bound to $\omega(\cE)$.

\subsection{Related works}\label{subsec:prior}
Singh and Vishnoi studied in~\cite{SV14} the bit-complexity of a class of entropy maximization problems, which corresponds to the dual of the present geometric programming problem, specialized to $N=1$. They provided polynomial-time approximability results. A key technical ingredient of their result is a bound on the optimal solution, based on a notion of conditioning, requiring the point $0$ to be ``deep inside'' the convex hull of $\cA$. More precisely,
denoting by $\eta$ the distance of $\zero$ to the boundary of the convex hull of $\cA$,
they show that the minimizer is of norm bounded in terms of $1/\eta$.
Here, we deal more generally with the minimization of the supremum of $N$ forms,
relying on the coercivity constant $\nu(\cA)$ in \eqref{def-coercive}.
When $N=1$, it can be checked that $\eta \geq \nu(\cA) \geq  \eta /\sqrt{n}$,
so $\eta$ and $\nu$ are condition numbers of a similar nature; $\nu(\cA)$ has the advantage
to coincide with the value of a linear program, leading to the explicit bound for the coercivity constant~\eqref{boundnu}.

B\"urgisser, Li, Nieuwboer and Walter also treat in~\cite{BLNW20} the minimization of the log of the Laplace transform of a nonnegative measure finitely supported in $\R^n$,  corresponding again to the
case $N=1$ in \eqref{deff(x)fj}.
Their approach uses a condition assumption, \cite[condition (b), page 1]{BLNW20},
similar to the one of~\cite{SV14}. 
They analyzed the complexity of the interior-point method, and provided a polynomial bound in the number of iterations (i.e., considering iteration complexity rather than bit-complexity). We check in~\Cref{sec:intmethod} that their interior-point approach carries over to maxima of $N$ forms. 
In Section 7.2 of  \cite{BDMWW24} B\"urgisser-Do\u{g}an-Makam-Walter-Wigderson consider the case $N=1$ under the assumption that $f$ is coercive.  They observe, as in Theorem \ref{mainthm-intro},  that one can find in polynomial time $\bx\in \Q^n$ such that $f(\bx)\le f_{\min}=f(\bx^{\star})+\varepsilon$.   In Example 7.4 
 authors show that $f(\x)-f_{\min}$ is exponentially small in an integer $N$ and doubly exponentially
small in the input bit-length, and yet $\|\x-\x^{\star}\|=1$. 
 This is similar in spirit to our Theorem \ref{th-cex}. 
 
Stratszak and Vishnoi
studied in~\cite{SV19} the ill-conditioned case, in which $\zero$ can be arbitrarily close to the boundary
of the convex hull of $\mathcal{A}$ (so that the $\eta$ parameter is arbitrarily small). They
exploited a notion of ``unary facet complexity'' of the Newton polytope of $f$,
``$\operatorname{ufc}$'', showing that there is a $\varepsilon$-minimizer
of norm $O(\operatorname{ufc}|\log\varepsilon|)$ (in which the
$O(\cdot)$ ignores terms depending on the dimension, cardinality of $\cA$, and
on the coefficients of $f$).
B\"urgisser, Li, Nieuwboer and Walter~\cite{BLNW20}
defined a geometric notion of ``facet gap'', which 
extends the
notion of unary facet complexity to instances with real rather than rational
exponents, and leads to refined bounds. We believe that the ``facet gap'' conditioning approach of~\cite{BLNW20} carries over to our setting with the supremum of $N$ forms, allowing us to obtain polynomial time results in the non-coercive case,
we leave this for further work.

The membership in P of the feasibility problem in geometric
programing is an open question.
Etessami, Steward, and Yannakakis showed in~\cite{Etessamietal}
that this problem is at least as hard
as a problem in arithmetic complexity, PosSLP (positive straight-line-program),
whose membership in P is a longstanding open question.
The fact that the Phase I approach requires exponential precision
to decide feasibility (\Cref{th-phase1})
may be seen as a further evidence of the difficulty of geometry programming.

The spectral radius of nonnegative tensors has received much attention, after the work
of Lim~\cite{Lim05} extending the notion of singular value to tensors, and of Qi~\cite{Qi05} studying the tensor eigenproblem.
To our knowledge, our results for tensors and maximum of homogeneous polynomials with nonnegative coefficients,  showing that the spectral radius and the logarithm of the eigenvector are polynomial time approximable, are new.
Ng,  Qi,  and G. Zhou developed in~\cite{NQZ09}
a power method, relying on the Collatz-Wielandt characterization~\eqref{minmaxchar},
to approximate the spectral radius and the associated eigenvector.
This method is well adapted to large scale instances,
but is does not lead to a polynomial bound. Indeed, parameterized bounds based on contraction properties in Hilbert metric can be obtained along the lines of~\cite{akianmfcs}, but even in the simple case of nonnegative matrices, i.e., tensors of degree $2$, the bounds obtained in this way are only pseudo polynomial.
An alternative approach, allowing one to find all the eigenvalues and eigenvectors of symmetric tensors using Lasserre relaxation, has been developed in~\cite{CDN14}.

Finally, we point out that our approach (especially the proof of~\Cref{mainthm})
is partly inspired by the proof of \cite[Theorem 15]{AGGG}, concerning a special geometric program arising from entropy games, in which
the bound on minimizers was obtained by an ad hoc argument
from Perron-Frobenius theory. Here, we develop general complexity
estimates relying the coercivity constant $\nu$.

\subsection{Short summary of the rest of the paper}\label{sec:sumcont}

In Section \ref{sec:coerc}, we discuss the coerciveness of $f(\x)$: necessary and sufficient conditions for $f$ to have compact sublevel sets: $f(\x)\le t$ for all $t\in\R$.
In Section \ref{sec:minpoints}, we discuss elementary properties of
the set of minimizers of $f$, under the coerciveness condition.  
In Section \ref{sec:boundKmin}, we  introduce a compact convex subset
$K(f,R)$ of the epigraph of $f$, which is useful to ``localize'' the minimization problem for $f$.

Section \ref{sec:ptcom} is the main technical section of this paper,
where we prove the polynomial time complexity of computing $f_{\min}$, the minimum of $f$, with an accuracy $\varepsilon$. This is equivalent to minimizing a linear function over $K(f,R)$.
Then, by applying the classical ellipsoid method, elaborated in \cite{GLS81,GLS88},  we prove our result.  

Section \ref{sec:intmethod} shows how to apply the IPM to find an approximation to $f_{\min}$, still working on the localized subset $K(f,R)$.   The construction of the barrier function for this set
is inspired by 
the barrier function introduced in \cite{BLNW20}.

In Section \ref{sec:minconfun}, we show the polynomial time approximability
of the spectral radius $\rho(\bF)$.
To establish this result, we exploit the coerciveness of the function $f$
restricted to a subspace of codimension $1$.  In subsection \ref{subsec:expswit} we provide an example of  a weakly irreducible $\bF:\R^{2n}\to\R^{2n}$, 
with the normalize eigenvector with some entries doubly exponential
in the input size,
and deduce that it is not possible in general to give a coarse approximation
in polynomial time of the positive eigenvector of a weakly irreducible partial symmetric tensor.  

In Section~\ref{sec:aplogeig}, we shows that if $\bF$ is strongly irreducible,
meaning that $D(\bF)$, the Jacobian of the map $\bF$, has positive off-diagonal entries for $\z\in\Rpp^{n}$, then an approximation of the positive log-eigenvector can be computed is polynomial time.

In Section~\ref{sec:symten}, we discuss the applications of our results to the following cases.
In Subsections  \ref{subsec:stenhp} and  \ref{subsec:pthmmupg}   we show that the maximum $\mu_d(g)$ of a nonnegative $g$ of degree $d$ form over the $\ell_d$ unit sphere,  is polynomial time computable.  
Subsection \ref{subsec:unhgr} discusses the spectral radius of uniform hypergraphs.  Proposition \ref{ubcliqueghr} gives a lower bound on the spectral radius
of uniform hypergraphs using the clique number of the hypergraphs.
This bound is closely related to the bounds in \cite{RBP09,XQ15}.  Since the spectral radius of a uniform hypergraph is polynomially computable, the above bound gives a computable upper bound for the clique number of a hypergraph. 
Subsection \ref{subsec:dihyper} discusses briefly the notion of dihypergraphs that is induced by nonnegative homogeneous maps of $\bF:\Rpp^{n}\to \Rpp^{n}$.
The spectral radius of a dihypergraph is $\rho(\bF)$.
Subsection \ref{subsec:quashommapspol} shows polynomial time computability of the spectral radius of weakly irreducible homogeneous posynomial maps, and $\mu_d(g)$ of $d$-homogeneous posynomial forms.
Subsection \ref{subsec:pcmupg} shows the polynomial time computability of $\mu_p(g)$ for $p\ge d\ge 2$  for $d$-homogeneous posynomial forms, where $p$ and $d$ are rational.  Subsection \ref{sec:pqmnorms} shows that the mixed norm $\|A\|_{p,q}$ for $A\in\Q^{m\times n}_{+}$ and $q\in\N, q\le p\in\Q$ is polynomial time computable.

\section{Coerciveness condition}\label{sec:coerc}
In this paper we use the following notations:
For $\x\in\R^n$, and $p\in[1,\infty]$, we denote
by $\|\x\|_p=(\sum_{i=1}^n|x_i|^p)^{1/p}$ the $\ell_p$ norm of $\bx$.
We denote by $\rB_p(\bc,r)$ the ball $\{\x\in\R^n: \|\x-\bc\|_p\le r\}$ centered at at point $\bc\in \R^n$. We denote by $\Rp$ and $\Rpp$ the sets of nonnegative and positive numbers, respectively. A similar notation applies to $\Q$ and $\Z$. 
For column vectors $\x\in\R^n, \y\in\R^m$,  we set $(\x,\y):=\begin{bmatrix}\x\\\y\end{bmatrix}\in\R^{n+m}$. 

We consider a function $f$ as in~\eqref{deff(x)fj}
Recall that $\cA=\cup_{j=1}^N \cA_j$, and set
\begin{equation}\label{defnu(A)}
\begin{aligned}
&t_{\min}(f):=\min\{\log f_{\ba,j}, \ba\in \cA_j, j\in[N]\},\\
&\nu(\cA)=\min_{\|\x\|_\infty=1}\max\{\ba^\top \x,  \ba\in\cA\} \enspace 
\end{aligned}
\end{equation}
\begin{lemma}\label{coerclemma}Let $f(\x)$ be defined by \eqref{deff(x)fj}.
The following conditions are equivalent
\begin{enumerate}[label=(\alph*)]
\item The set $\{\x\mid f(\x)\le t\}$ is compact for all $t\in\R$. 
\item The set $\{\x\mid f(\x)\le f(\x_1)\}$ is compact for some $\x_1\in\R^n$.
\item The following condition holds:
\begin{equation}\label{coercon}
  \{\x\in\R^n,\ba^\top \x\le 0 \textrm{ for } \ba\in\cA\}= \{\0\}
  \enspace . 
\end{equation}
\item The convex hull of $\cA$ contains a neighborhood of $\0$.
\item
We have $\nu(\cA)>0$ and 
\begin{equation}\label{nucond}
f(\x)\ge t_{\min}(f)+\nu(\cA)\|\x\|_\infty.
\end{equation}
\end{enumerate}
\end{lemma}
\begin{proof}
(a) $\Rightarrow$ (b): Trivial.

\noindent
(b) $\Rightarrow$ (c).  Assume to the contrary that there exist $\x_0\in\R^n\setminus\{\0\}$ such that $\ba^\top\x_0\le 0$ for all $\ba\in\cA$.  Then for each $t>0$ we have $\ba^\top(t\x_0)\le 0$ for all $\ba\in\cA_j$ and $j\in[N]$.  Therefore $f(t\x_0+\x_1)\le f(\x_1)$.  Hence the sublevel set $f(\x)\le f(\x_1)$ is not compact.

\noindent
(c) $\Rightarrow$ (d).  Let $M\coloneqq |\cA|$,  and  $A=[a_{ik}]\in\R^{n\times M}$,  where the columns of $A$ are the vectors $\ba\in\cA$.   Hence the system of inequalities $\ba^\top \x\le \0, \ba\in\cA$ is equivalent to $A^\top \x\le \0$.    Denote by $\rK\subset \R^n$ the polytope spanned by $ \cA$.  
Assume that $\0\not\in \rK$ or $\0\in\partial \rK$.   Then, by the weak-separation theorem (see e.g.\ Theorem~11.6 of~\cite{rockafellar}),
there exists $\x_0\in\R^n\setminus\{\0\}$ such that $A^\top \x_0\le \0$,  which contradicts \eqref{coercon}. Hence, $0$ is an interior point of $\rK$.

\noindent
(d) $\Rightarrow$ (c).   Assume to the contrary that there exists $\x\in\R^n\setminus\{\0\}$ such that $A^\top \x\le \0$.   As a neighborhood of $\0$ is contained in $K$, there exists a probability vector $\z\in\R^m$ such that $ A\z=r\x$ for some $r>0$.
Therefore $\z^\top A^\top\x=r\x^\top\x>0$ contradicting the assumption that $A^\top \x\le \0$.
 
\noindent
(c) $\Rightarrow$ (e).  Assume to the contrary that $\nu(\cA)\le 0$.
Hence, there exists $\x_0\in\R^n, \|\x_0\|_{\infty}=1$ such that  $\ba^\top\x_0\le 0$ for all $\ba\in\cA$.   This contradicts (c).  Hence $\nu(\cA)>0$.   

Consider now a vector $\x$ such that $\|\x\|_{\infty}=r>0$.  
Then
\begin{equation}\label{lbfnu}
\begin{aligned}
&f(\x)\ge  t_{\min}(f)+\|\x\|_{\infty}\max\{\ba^\top \big((1/r) \x\big),
\ba\in \cA\}\\
&\ge
t_{\min}(f)+\nu(\cA)\|\x\|_\infty.
\end{aligned}
\end{equation}

\noindent
(e) $\Rightarrow$ (a).  Let $t\in \R$.  Clearly, if 
\begin{equation*}
\|\x\|_{\infty}>\frac{t-t_{\min}(f)}{\nu(\cA)},
\end{equation*}
then $f(\x)>t$.  Hence $\{\x\mid f(\x)\le t\}$ is a compact set.
\end{proof}
The condition $(d)$ is standard, compare for instance with \cite[condition (b) page 2]{BLNW20} dealing with the special case $N=1$.
\begin{corollary}\label{corcoerc}  Assume that $f$ is defined by \eqref{deff(x)fj} is coercive.  Then the cardinality of the set $\cA$ is at least $n+1$.
\end{corollary}
The coercivity constant $\nu(\cA)$ is the best constant such that
\[
\max\{\ba^\top \x,  \ba\in\cA\}  \geq \nu(\cA) \|\x\|_\infty, \forall \bx\in\R^n \enspace .
\]
By convex duality, for any conjugate exponents $p,q\in [1,\infty]$, so that $1/p+1/q=1$,
we have that
$\max\{\ba^\top \x,  \ba\in\cA\}   \geq r \|\x\|_q$ holds for all $\bx\in \R^n$
iff $B_p(\zero,r) \subset \operatorname{conv}(\cA)$.
I.e., setting
$r_p(\cA):= \max\{ r\geq 0\mid B_p(\zero,r) \subset \operatorname{conv}(\cA)\}$,
we see that $\nu(\cA)=r_1(\cA)$. Moreover, $r_2(\cA) \leq r_1(\cA)\leq \sqrt{n}r_2(\cA)$.
When $N=1$, the quantity $r_{2}(\cA)$, i.e. the Euclidean distance of $\zero$
to the boundary of $\operatorname{conv}(\cA)$, has been used in~\cite{SV14,BLNW20} to parameterize
the complexity of the geometric programming problem~\eqref{e-geo}.

\section{Points of minimum}\label{sec:minpoints}
The Hessian of the log-Laplace transform of a nonnegative measure
can be interpreted as the covariance matrix
of a random variable with the same support as this measure
(see e.g. proof of~\cite[Prop~7.2.1]{brazitikos}). Hence,
as soon as this support contains an affine generating family,
this Hessian is positive definite. In this way,
we readily deduce the following
standard result (the second statement appeared in~\cite[Prop.~1]{CaGaPo:18}).
\begin{lemma}
  \label{nsconv}
  Let $\vec{E}$ be an affine vector space generated by $\cA$.  Then the restriction of $f$ to $\vec{E}$ is strictly convex.   In particular,  the function $f$ is strictly convex on $\R^n$ if, and only if,  the vectors $\ba\in \cA$ constitute an affine generating family of $\R^n$.
\hfill\qed
\end{lemma}

Denote by $\rK_{\min}(f)\subseteq\R^n$ the set of minimal points of $f$ given by 
\eqref{deff(x)fj}.
Then $f$ does not achieve its minimum in $\R^n$ if and only if $\rK_{\min}(f)=\emptyset$.
The following is a standard property of the set of minimizers of a coercive and convex function. 

\begin{lemma}[{\cite[Prop.~1.2 p.~35]{ekelandtemam}}]\label{minplemma1}
Suppose that $f$ is coercive. Then, $\rK_{\min}(f)$ is a non-empty compact convex set.
\end{lemma}

We now discuss the characterization of $\rK_{\min}(f)$.  For each $\x\in\R^n$ define the nonempty set $\cB(\x)\subseteq [N]$:
\begin{equation}\label{uniqmin1}
\begin{aligned}
f_j(\x)=f(\x) \textrm{ for } j\in\cB(\x),\\
f_{j}(\x)<f(\x) \textrm{ for } j\not\in \cB(\x).
\end{aligned}
\end{equation}
\begin{lemma}\label{uniqmin}
Assume that $\rK_{\min}(f)\ne \emptyset$.  Then $\x_0\in \rK_{\min}(f)$ if and only if the following condition holds:   the system of linear inequalities 
\begin{equation}\label{uniqmin2}
\nabla f_{j}(\x_0)^\top \y<0,  j\in\cB(\x_0)
\end{equation}
has no solution.

Moreover, the following conditions are equivalent 
\begin{enumerate}[label=(\alph*)]
\item The dimension of $\rK_{\min}(f)$ is at least $1$. 
\item Assume that $\x_0\in\rK_{\min}(f)$.  There exists $\y\in\R^n\setminus\{\0\}$ such that  the inequality 
\begin{equation}\label{uniqmin3}
\nabla f_{j}(\x_0)^\top \y\le 0,  j\in\cB(\x_0)
\end{equation}
holds.  Furthermore,  let 
\begin{equation*}
\cC(\x_0,\y)=\{j\in \cB(\x_0),   f_{j}(\x_0)^\top \y=0\}.
\end{equation*}
Then for each $j\in \cC(\x_0,\y)$ one has the equality
\begin{equation}\label{uniqmin4}
\ba^\top \y=0,  \,\ba\in\cA_j,\,j\in \cC(\x_0,\y).
\end{equation}
\end{enumerate}

More precisely,  $\dim\rK_{\min}(f)=d\ge 1$
if and only if 
the following conditions hold:
For each $\x_0\in\rK_{\min}(f)$ the set of 
$\y\in \R^n\setminus\{ \0\}$ satisfying the condition (b) together with $\0$ is a cone $\rK(\x_0)\subseteq\R^n$ of dimension $d$.
\end{lemma}
\begin{proof}
Assume that $\x_0\in K_{\min}(f)$.   Suppose to the contrary that there exists $\y\in\R^n$ that satisfies \eqref{uniqmin2}. Clearly, $\y\ne \0$.  Take the directional derivative of each $f_{j},j\in\cB(\x_0)$ at $\x_0$ in the direction of $\y$ to deduce that $\x_0\not\in\rK_{\min}(f)$, which contradicts our assumption.  

Assume that the system of inequalities  \eqref{uniqmin2} is unsolvable.  We claim that $\x_0\in \rK_{\min}(f)$.  Assume to the contrary that there exists $\x_1\in \R^n$ such that 
$f(\x_1)<f(\x_0)$.   In particular, 
\begin{equation*}
\tilde f(\x_1):=\max\{f_{j},(\x), j\in\cB(\x_0)\}\le f(\x_1)<f(\x_0).
\end{equation*}
Set $\y=\x_1-\x_0$.    The assumption that \eqref{uniqmin2} is unsolvable means that
there exists $j\in \cB(\x_0)$ such that $\nabla f_{j}(\x_0)^\top \y\ge 0$.
Consider the convex function $g(t)=f_{j}(\x_0+t\y)$.  Our assumption yields that $g'(0)\ge 0$.  Since $g(t)$ is smooth and convex we deduce that  $g'(t)\ge 0$ for $t\ge 0$.  Hence $f_{j}(\x_1)=g(1)\ge g(0)=f_{j}(\x_0)=f(\x_0),$ which contradicts our assumption that $f(\x_1)<f(\x_0)$.

Assume the assumption (b) of the Lemma.
Then there exists $\x_0\in\rK_{\min}(\y)$ and $\y\ne \0$ which satisfy\eqref{uniqmin3}  and \eqref{uniqmin4}.   Let $g_j(t)=f_{j}(\x_0+t\y)$ for $j\in [N]$.  Assume that
$j\in \cB(\x_0)\setminus \cC(\x_0,\y)$.  Then $h_j'(0)<0$.  Therefore, there exists $r>0$ such that $h_j(t)\le h_j(0)=f(\x_0)$ for $t\in[0,r]$.   Assume that $j\not\in \cB(\x_0)$.
As $f_{j}(\x_0)<f(\x_0)$ for  we can assume that $r>0$ is small enough so that $h_j(t)\le f(\x_0)$ for $t\in[0,r]$.   The assumption that the system \eqref{uniqmin2} is unsolvable implies that $\cC(\x_0,\y)$ is nonempty.
The assumption \eqref{uniqmin4} means that $h_j(t)= f(\x_0)$ for $t\in\R$ and $j\in\cC(\x_0,\y)$.  Hence $f(\x)=f(\x_0)$ for $\x\in [\x_0, \x_0+r\y]$,  and $\dim K_{\min}\ge 1$.

To show part (a) of the lemma, it is enough to prove the last claim of the lemma.
Assume now that $\dim\rK_{\min}(f)=r\ge 1$.  Let $\x_0\in\rK_{\min}(f)$ let $\rK(\x_0)$ be  the cone generated by $\z-\x_0$ for  $\z\in \rK_{\min}(f)\setminus\{\x_0\}$.   Clearly,  $\y \in\rK(\x_0)\setminus\{\0\}$ is of the form $s(\z-\x_0)$ for some $s>0$ and $\z \in\\rK_{\min}(f)$.  Without loss of generality we can consider all $\y\ne \0$ of the above form where $s=1$.   It is straightforward to show that $\y$ satisfies \eqref{uniqmin3}.
Assume that $j\in\cC(\x_0,\y)$.  Let $h_j(t)=f_{j}(\x_0+t\y)$.  As $h_j(t)\le h_j(0)=f(\x_0)$, and $h_j'(0)=0$ we deduce that $h_j(t)=h_j(0)$ for $t\in[0,1]$.  Thus $h_j(t)$ is not strictly convex.  Apply Lemma \ref{nsconv}  to deduce \eqref{uniqmin4}.

Vice versa,  assume that for $\x_0\in\rK_{\min}(f)$ the set of 
$\y\in \R^n\setminus\{ \0\}$ satisfying the condition (b) together with $\0$ is a cone $\rK(\x_0)\subseteq\R^n$ of dimension $d$.  Then the arguments in the beginning of proof of this lemma show that $\dim \rK_{\min}(f)\ge d$.  If $\dim \rK_{\min}(f)> d$ we will deduce that $\dim \rK(\x_0)>d$, which contradicts our assumptions.
\end{proof}

It is easy to give an example showing that $\dim K_{\min}(f)$ can take any value in $\{0,\ldots,n\}$ by taking each function
$f_{\cA_j}$ to be affine, i.e.  $\cA_j=\{\ba_j\}$ for $j\in[N]$.   In this case
$K_{\min}(f)$ is a polyhedron.
For example, 
\begin{equation}\label{exampKmin}
\begin{aligned}
&f_{1}=0,  f_{2i}=-x_i-b_i,  f_{2i+1}=x_i-c_i,  b_i\le c_i, i\in[n]\\ &f(\x)=\max(0,-x_1-b_1,\ldots,x_n-c_n) \textrm{ for } \x=(x_1,\ldots,x_n)^\top\in\R^n,\\
&\rK_{\min}(f)=[b_1,c_1]\times\cdots\times[b_n,c_n].
\end{aligned}
\end{equation}
Note however that $K_{\min}(f)$ need not be a polyhedral set. To see this,
let
\begin{equation*}
f_{1}=0, f_{2}=-1-x_1,f_{3}=-1-x_2,  f_{4}=\log(\exp(x_1)+\exp(x_2).
\end{equation*}
Observe that $f$ is coercive.  Furthermore,
\begin{equation*}
K_{\min}(f)=\{\bx\in \R^2\mid \exp(x_1)+\exp(x_2)\leq 1, -1\le  x_1,-1\le x_2\}.
\end{equation*}

Observe that $K_{\min}(f)$ is a compact convex set whose boundary is included in the curve $f_{\cA_i}=e^{f_{\min}}$, showing that the geometry of $K_{\min}$ can be non-trivial.
\section{A bound on the set of minimizers $\rK_{\min}(f)$ for a coercive $f$}\label{sec:boundKmin}
\begin{lemma}\label{boundKmin}  Assume that $f$ is coercive.   Let $t_{\min}(f)$and $\nu(\cA)$ be defined by \eqref{defnu(A)}.   Suppose that
\begin{equation}\label{defR}
R\ge\frac{f(\0)- t_{\min}(f)}{\nu(\cA)}.
\end{equation}
Then $K_{\min}(f)\subseteq \rB_{\infty}(\0,R)$.
\end{lemma}
\begin{proof} Clearly,  $\min\{f(\x),\x\in\R^n\}\le f(\0)$.  Use \eqref{nucond} to deduce that $f(\x)> f(\0)$ for $\|\x\|_{\infty}> R$.  Hence,  $K_{\min}(f)\subseteq \rB_{\infty}(0,R)$.
\end{proof}

Owing to~\eqref{defR}, it is desirable to control
$\nu(\cA)$. We next show that $\nu(\cA)$
can be computed by solving $2n$ linear programs.
\begin{lemma}\label{nuchar}    Assume that the condition \eqref{coercon} holds.
For $i\in[n]$ define 
\begin{equation}\label{defnui}
\begin{aligned}
&\nu_{i,+}(\cA)=\min\{t,  \ba^\top \x- t\le 0,  \ba\in\cA, \\
&\x=(x_1,\ldots,x_n)^\top\in\R^n, x_i=1, -1\le x_j\le 1, j\in[n]\setminus\{i\}\},\\
&\nu_{i,-}(\cA)=\min\{t,  \ba^\top \x- t\le 0,  \ba\in\cA, \\
&\x=(x_1,\ldots,x_n)^\top\in\R^n, x_i=-1, -1\le x_j\le 1, j\in[n]\setminus\{i\}\},
\end{aligned}
\end{equation}
Then 
\begin{equation}\label{nuchar1} 
\nu(\cA)= \min_{i\in[n]} \min(\nu_{i,+}(\cA),\nu_{i,-}(\cA)).
\end{equation}
\end{lemma}
\begin{proof} 
Observe that
\begin{equation*}
\nu_{i,+}(\cA)=\min_{x_i=1, -1\le  x_j\le 1, j\ne i}\max_{\ba\in \cA}(\ba^\top \x)\}.
\end{equation*}
We next observe that $\nu_{i,+}(\cA)\ge \nu(\cA)$.
Indeed, the minimum in the characterization $\nu_{i,+}(\cA)$ is stated for all $\|\x\|_{\infty}=1$.   Similarly,  $\nu_{i,-}(\cA)\ge \nu(\cA)$.   Hence the left hand side of \eqref{nuchar1} is not bigger than the right hand side of  \eqref{nuchar1}.

Assume that  the minimum $\nu(\cA)$ is achieved at $\x^{\star}=(x_1^{\star},\dots,x_n^\star)^\top$.   Observe that $\x_i^\star\in\{-1,1\}$ for some $i\in[n]$.  Hence,  $\nu(\cA)\in\{\nu_{i,-}(\cA),\nu_{i,+}(\cA)\}$.
\end{proof}

The previous characterizations of $\nu_{i,-}(\cA)$ and $\nu_{i,+}(\cA)$ involve the solutions to linear programming problems.
Therefore, the value $\nu(\cA)$ can be computed in polynomial time.
Furthermore, we next give a simple explicit bound for $\nu(\cA)$,
supposing that the vectors of $\cA$ are integers.
\begin{lemma}\label{nulowest} Assume that $\cA\subset \Z^n$.  Suppose that the
  coerciveness condition \eqref{coercon} holds.   Let $|\cA|=M(\ge n+1)$,
and assume that 
\begin{equation}\label{unAjar}
\cA=\{\bb_1,\ldots,\bb_M\}, \quad\|\bb_1\|_2^2\ge \cdots\ge \|\bb_M\|_2^2. 
\end{equation}
Let $\nu(\cA)$ be defined by \eqref{defnu(A)}.
Then
\begin{equation}\label{nulowest1} 
\begin{aligned}
\nu(\cA)\ge \frac{1}{\prod_{i=1}^n \sqrt{\|\bb_i\|^2 +1}}.
\end{aligned}
\end{equation}
\end{lemma}\begin{proof}  Assume that 
\begin{equation*}
t^{\star}=\nu(\cA)=\max_{\bb_k, k\in[M]} \bb_j^\top \x^\star, \x^{\star}=(
x_1^{\star},\cdots,x_n^{\star})^\top, \|\x^\star\|_{\infty}=1.
\end{equation*}
We choose $\x^{\star}$ with the maximum coordinates equal to $\pm 1$.  
Recall that $t^{\star}>0$.  Suppose that $\x^{\star}\in\{-1,1\}^n$.  Then $t^\star$ is a positive integer, hence $t^\star\ge 1$.  Clearly, the inequality of \eqref{nulowest1} holds.
We now assume that $\x^{\star}$ has a minimum number of coordinates $l\in[n-1]$ which are not equal to $\pm 1$.  Among all such $\x^{\star}$ we choose one for which
the set $\cI=\{k\in[M], \bb_k^\top \x^\star=t^{\star}\}$ is maximal.  Let $\y\in\R^l$ be the vector whose coordinates are all coordinates of $\x^\star$ that are not equal to $\pm 1$.  Then $\z=(\y^\top,t^\star)^\top \in\R^{l+1}$ are satisfying $|\cI|$ linear equations of the form $\bc_k^\top \y-t^{\star}=d_k$ for $k\in \cI$.   We claim that the system $\bc_k^\top \w-t=d_k$ for $k\in \cI$ has a unique solution $(\y^\top,t^\star)^\top$.  Otherwise there exists a nontrivial solution $(\w,t)\in \R^{l+1}$ satisfying $\bc_k^\top \w-t=0$ for $k\in\cI$.  If $t\ne 0$ it is straightforward to deduce that $t^\star>\nu(\cA)$, which contradicts our assumption. 
Assume that $t=0$.  As $\w\ne 0$ we deduce that either $l$ is not minimal or $\cI$ is not maximal, contrary to our assumptions.  Hence the system $\bc_k^\top \w-t=d_k$ for $j\in \cI$ has a unique solution $(\y,t^\star)$.  Thus $|\cI|\ge l+1$ and we can choose $l+1$ equation from the system $\bc_k^\top \w-t=d_k$ which gives the unique solution $(\y,t^{\star})^\top$.  Next observe that $d_k$ are integers and $\bc_k$ are
are obtained for $\ba_k$ by deleting the coordinates of $\ba_k$ corresponding to the coordinates of $\x^{\star}$ equal or $\pm 1$.
Next we use Cramer's rule to determine the value of $t^{\star}$. Note the $t^{\star}>0$ is a quotient of two determinants with integer entries.  As $t^{\star}>0$ we can  take the absolute values of these two determinants.  Clearly the absolute value of the numerator is at least $1$.  The maximum value of the denominator is estimated from above by using the Hadamard determinant inequality on the rows of the denominator.
Recall that each row of the denominator is a vector of the form $(\bc_k^\top,-1)$.
Note that 
\begin{equation*}
\|(\bc_k^\top,-1)\|_2=\sqrt{\|\bc_k\|_2^2+1}\le \sqrt{\|\ba_k\|^2+1}.
\end{equation*}
Hence, the absolute value of the denominator is at most $\prod_{i=1}^n \sqrt{\|\bb_i\|^2 +1}$.  This proves the inequality in \eqref{nulowest1}.
\end{proof}
\section{Reduction to a minimization problem on a well conditioned set}
In order to apply the ellipsoid method to the minimization of $f$, we next
show that the latter problem is equivalent to the minimization of the
``altitude'' $t$ on a suitable compact convex section $K(f,R)$ of the epigraph of $f$. Furthermore, we show that this compact convex set is sandwitched bewteen
inner and outer balls with well controlled radii.
\begin{lemma}\label{propKfR}
  Assume that $R\ge 0$.   Let
\begin{equation*}
\begin{aligned}
&a(f)=\max\{\|\ba\|_1, \ba\in  \cA\},\qquad
t_{\max}(f,R)=f(\0)+(a(f)+1)R,  \quad \bar t(f,R)=f(\0)+\frac{(a(f)+1)R}{2},\\
&R_{\max} =(a(f)+1)R+f(\0)-t_{\min}(f),\\
&\rK(f,R)=\{(\x,t), \x\in\R^n, t\in\R, f(\x)\le t, \, \|\x\|_{\infty}\le R, t\le t_{\max}(f,R)\}.
\end{aligned}
\end{equation*}
Then $K(f,R)$ is a compact convex set satisfying
\begin{equation}\label{contcondK(f,R)}
B_{\infty}((\0, \bar t(f,R)),R/2)\,\subset\, K(f,R)\,\subset \, B_{\infty}((\0, \bar t(f,R)),R_{\max}).
\end{equation}
Assume that $f$ is coercive, and $R$ satisfies \eqref{defR}. 
Then
\begin{equation}\label{minfchar}
f_{\min}=\min\{t, (\x,t)\in \rK(f,R)\}.
\end{equation}
\end{lemma}
\begin{proof} Clearly, $K(f,R)$ is compact.  As $f$ is convex it follows that $K(f,R)$ is convex.   
  As $\exp(\ba^\top\x)\le 
\exp(\|\ba\|_1\|\x\|_{\infty})$ we deduce
\begin{equation*}
f(\x)\le f(\0) +\max\{\|\ba\|_1, \ba\in\cA\}\|\x\|_{\infty}.
\end{equation*}
Combine the above inequality with \eqref{nucond} to deduce
\begin{equation}\label{uplowfestR}
t_{\min}(f)\le f(\x)\le f(\0)+ a(f)R \textrm{ for } \|\x\|_{\infty}\le R.
\end{equation}
Assume that $\|\x\|_{\infty}\le R/2$.  The above inequality shows that $f(\x)\le f(\0)+a(f)R/2$.  Hence,
\begin{equation*}
\|\x\|_{\infty}\le R/2, f(\0)+a(f)R/2\le t\le t_{\max}(f,R)\Rightarrow (\x,t)\in \rK(f,R).
\end{equation*}
As 
\begin{equation*}
-R/2\le t-\bar t(f,R)\le (a(f)/2 +1)R
\end{equation*}
we deduce the first containment of  \eqref{contcondK(f,R)}.  
Assume next that $\|\x\|_\infty\le R$.  Then \eqref{uplowfestR} yields that $t_{\min}(f)\le f(\x)\le f(\0)+a(f)R$.
Thus, for $(\x,t)\in\rK(f,R)$ we have the inequalities
\begin{equation*}
-((a(f)+1)R/2 +f(\0)-t_{\min}(f))\le t-\bar t(f,R)\le (a(f)+1)R/2.
\end{equation*}
This yields that second containment of  \eqref{contcondK(f,R)}. 

We now show the characterization  \eqref{minfchar}.   The inequality  \eqref{uplowfestR} yields that for each $\x\in \rB_{\infty}(\0,R)$ the point $(\x,f(\x))\in \rK(f,R)$.
Clearly, for a fixed $\x\in \rB_{\infty}(\0,R)$ the minimum of $t$ in $K(f,R)$ is $f(\x)$.  
As $R$ satisfies \eqref{defR}, we deduce \eqref{minfchar}.  
\end{proof}
In what follows we assume the nontrivial case
\begin{equation}\label{ncase}
f(\0)>t_{\min}(f) (\Rightarrow  \lceil f(\0)\rceil- \lfloor t_{\min}(f)\rfloor >0).
\end{equation}
(Otherwise $f_{\min}=f(\0)=t_{\min}(f)$.)
In this case $R(f)\in\N$, and $R(f)$ satisfies the condition \eqref{defR}.

Recall that
\begin{equation}\label{normin}
\begin{aligned}
\|\x\|_\infty\le \|\x\|_2 \le \|\x\|_1,  \|\x\|_2\le \sqrt{n}\|\x\|_{\infty},  \|\x\|_1\le \sqrt{n}\|\x\|_2 \textrm{ for } \x\in\R^n,\\
\rB_1(\bc,r)\subset \rB_2(\bc,r)\subset \rB_{\infty}(\bc,r),\\
\rB_\infty(\bc,r)\subset \rB_2(\bc,\sqrt{n}r), \quad \rB_{2}(\bc,r)\subset \rB_1(\bc,\sqrt{n}r).
\end{aligned}
\end{equation}
\begin{lemma}\label{contlemma}  Assume that conditions \eqref{cAjfbacond} and 
\eqref{coercon} hold.  
Let
\begin{equation}\label{defR(f)R2}
\begin{aligned}
&R(f)=(\lceil f(\0)\rceil- \lfloor t_{\min}(f)\rfloor)\lceil\sqrt{\prod_{k=1}^n (\|\bb_k\|^2_2+1)}\,\rceil,\\
&R_{2}(f) =\lceil\sqrt{n+1}\rceil\big((a(f)+1)R(f)+\lceil f(\0)\rceil -\lfloor t_{\min}(f)\rfloor\big)
\end{aligned}
\end{equation}
Then
\begin{equation}\label{contcondK(f,R(f))}
\begin{aligned}
B_{2}((\0, t(f,R(f))),R(f)/2) \,\subset\, K(f,R(f))\,\subset\, B_{2}((\0, t(f,R(f))),R_2(f)),
\end{aligned}
\end{equation}
\begin{equation}\label{minfchar1}
f_{\min}=\min\{t, (\x,t)\in \rK(f,R(f))\}.
\end{equation}
\end{lemma}
\begin{proof} The arguments of the proof of Lemma \ref{propKfR} yield
\begin{equation}\label{K(f)contain}
\rB_{\infty}((\0,t(f)),R(f)/2)\subset \rK(f)\subset \rB_{\infty}((\0,t(f)),R_2(f)/(\lceil\sqrt{n}\rceil).
\end{equation}
Use the inequalities \eqref{normin} to deduce \eqref{contcondK(f,R(f))}.
Combine Lemma \ref{nulowest} with the proof of \eqref{minfchar} to deduce \eqref{minfchar1}.
\end{proof}

\section{Polynomial time computation of $f_{\min}$}
\label{sec:ptcom}
The main result of this section is:
\begin{theorem}\label{mainthm}
  Let $f$ be defined by \eqref{deff(x)fj}, and
  assume that the coercivity condition
\eqref{coercon} holds.
Assume that $\varepsilon$ is a given positive rational number.  Then one can compute a
rational vector $\x^\star$ such that $f(\x^\star)\leq f_{\min}+\varepsilon$
in polynomial time in $|\log\varepsilon|+\<f>$.
\end{theorem}
The proof of this theorem is based on the ellipsoid method \cite{GLS81}.
The main point in applying this result is to get
an a priori bound on the approximate
minimizers, and this will be done thanks to~\eqref{defR}.

We assume that $f$ is coercive, and recall
that
\begin{equation*}
 \cA_j\subset \Q^n, \,f_{\ba,j}\in\Qpp\textrm{ for }\ba \in\cA_j,   j\in[N].  
 \end{equation*}
Let $\delta$ denote the common denominator of all rational numbers in each coordinate of $\ba\in\cA_j$ for $j\in[N]$. Consider the change of variable $\bx  = \delta \bx'$, $g_j(\bx'):=f_j(\delta \x)$, $j\in[N]$ and $g:=\sup_{j\in[N]}g_j$. Then, $g_j$ has the same form as $f_j$, the set
$\cA_j$ being replaced by $\delta \cA_j$. 
Moreover, the size
$\<g>$ is polynomially bounded in the size $\<f>$,
and the size of $\varepsilon$-minimizers of $f$ is polynomially bounded
in the size of $\varepsilon$-minimizers of $g$.
So, we can assume
without loss of generality that
\begin{align}
\label{cAjfbacond}
\cA_j\subset \Z^n, \qquad \text{ for all } j\in[N] \enspace .
\end{align}

%

%
We now recall the main complexity result of the ellipsoid method \cite{GLS81}.
It relies on the following notion.
\begin{definition}\label{defwso}  Let $\rK\subset \R^n$ be a compact convex set in $\R^n$.
A {\it weak separation oracle} for $\rK$, taking an input number $\varepsilon\in \Qpp$ and $\y\in\Q^n$, and concludes one of the following:
\begin{enumerate}[label=(\roman*)]
\item Asserting that $\y$ is at Euclidean distance at most $\varepsilon$ from $\rK$.
\item Finding an {\it approximate separating half space of precision $\varepsilon$}, i.e. , 
a linear form $\phi(\x)=\bc^\top \x$, with $\bc\in\Q^n$,  and $\|\bc\|_2\ge 1$, such that for every $\x\in\rK$ the inequality $\phi(\x)\le \phi(\y)+\varepsilon$ holds. 
\end{enumerate}
\end{definition}

For $r\in\Q$ denote by $\langle r \rangle$ the number of bits needed to encode $r$,
under standard binary recording.
For instance,  if $r$ is a nonzero integer then $\langle r\rangle =\lceil \log_2 |r|\rceil +1$,  and if $r=p/q\in\Q$ then $\langle r\rangle=\langle p\rangle +\langle q\rangle$, if $\br=(r_1,\ldots,r_n)^\top\in\Q^n$ then $\langle \br\rangle=\sum_{i=1}^n \langle r_i\rangle$, and if $\psi(\x)=\br^\top \x$ a linear rational form on $\R^n$, then $\langle \psi\rangle=\langle \br\rangle$.
In what follows,  the length of an input refers to the binary coding.

The ellipsoid method can be used to find the minimum of a linear rational form $\psi$ on the compact convex set $\rK\subset\R^n$ within $\varepsilon$ precision.  This means finding $\x^\star\in\Q^n$ such that $\dist_2(\x^\star,\rK)\le \varepsilon$ and
$\psi(\x^\star)\le \min_{\x\in\rK} \psi(\x) +\varepsilon$.  (Here $\dist_2(\x^\star,\rK)=\min_{x\in\rK}\|\x^\star-\x\|_2$.)

A main assumption on $\rK$ is the containment 
\begin{equation}\label{contassum}
\rB_2(\bc, R_1)\subseteq \rK\subset\rB_2(\bc,R_2), \quad \bc\in\Q^n, R_1\le R_2, R_1,R_2\in\Qpp.
\end{equation}
In this case, the size of the input approximate minimization is measured by 
$\langle \psi\rangle +\langle \bc\rangle +\langle R_1\rangle +\langle R_2\rangle +\langle \varepsilon\rangle$.

It is shown in \cite[Theorem 3.1]{GLS81} that if a compact convex set $\rK$ satisfying \eqref{contassum} admits a polynomial time weak separation oracle, the ellipsoid method computes an $\varepsilon$-approximation of the minimization problem of a linear rational form in a polynomial time in the size of the input.

\begin{proposition}\label{weakseprop}
Assume that conditions \eqref{cAjfbacond} and 
\eqref{coercon} hold.    Let $\rK(f,R(f))$ be defined as in Lemma~\ref{propKfR}, where $R(f)$ is given by \eqref{defR(f)R2}.  Then $\rK(f,R(f))$ admits 
a weak separation oracle which for a precision $\varepsilon$ runs in polynomial time in $\<f>$ and $|\log\varepsilon|$.
\end{proposition}

To prove this proposition we recall the following
consequence of a result of Borwein and Borwein on the approximation
of the $\log$ and $\exp$ maps.
\begin{lemma}[See~{\cite{borwein}}]
  \label{approxexlog}
Let $\varepsilon\in\Qpp$ and $t\in\Q$.  
\begin{enumerate}[label=(\roman*)]
\item For $t\le 0$,  a rational $\varepsilon$-approximation of $\exp(t)$ can be computed in a polynomial time in $\langle t\rangle$ and $\langle \varepsilon\rangle$.
\item For $t>0$ a rational $\varepsilon$-approximation of $\log t$  can be computed in a polynomial time in $\langle t\rangle$ and $\langle \varepsilon\rangle$.
\end{enumerate}
\end{lemma}
\begin{lemma}
  We can compute in polynomial time a $\varepsilon$ approximation of the value $f(\bx)$ of the function $f$ in~\eqref{deff(x)fj} at a point $\bx\in\Q^n$.
\end{lemma}
\begin{proof}
  To do so, we avoid the direct computation of the terms $\exp(\<\ba,\bx>)$.
Instead, we first compute, for all $i\in [N]$, 
\[
\bar{f}_i(\bx) = \max_{\ba \in\cA_i} \<\ba,\bx>
\]
and use the fact that
\[
f_i(\bx) = \bar{f}_i(\bx)+ \log \big (\sum_{\ba\in\cA_i} f_{\ba,i} \exp(\<\ba,\bx>-\bar{f}_i(\bx))\big)
\]
so that the exponential map is only evaluated at nonpositive points,
in accordance with the restriction of the domain in~\Cref{approxexlog}, (i).
\end{proof}
\begin{proof}[Proof of Proposition~\ref{weakseprop}]
   Assume that $\varepsilon\in \Qpp$ is given. For $\bc\in\R^{n+1}$ denote by $\phi_{\bc}$ the linear functional on $\R^{n+1}$ 
given by $\bc^\top\w$, for $\w\in\R^n$.
Let $\z=(\y^\top,s)^\top\in \Q^{n+1}$.

\begin{enumerate}[label=(\alph*),wide, labelwidth=!, labelindent=0pt]
\item
Suppose that $\|\y\|_{\infty} \ge R(f)$.  Assume that $|y_i| \ge R(f)$ for some $i\in[n]$.  Denote by $\sgn y_i\in\{-1,1\}$
the sign of $y_i$.  Denote by $\be_j\in\R^{n+1}$ a standard $j$-th basis element in $\R^{n+1}$,
whose $j-th$ coordinates is $1$ and all other coordinates are $0$ for $j\in[n+1]$.   Set $\bc=\sgn(y_i)\be_i$.  Note that $\|\bc\|_2=1$.   
Clearly, $\phi_{\bc}(\w)\le R(f)$ for $\w\in\rK(f,R(f))$, and $\phi_{\bc}( \z)\ge R(f)$.  Hence $\phi_{\bc}$ separates between $\rK(f,R(f))$ and $\z$.
\item
Suppose that $\|\y\|_{\infty}<R(f)$.  Assume first that $s\ge t_{\max}(f)$.  Clearly, $\phi_{\be_{n+1}}(\w)\le t_{\max}(f)$, and $\phi_{\be_{n+1}}(\z)\ge t_{\max}(f)$. Hence $\phi_{\be_{n+1}}$ separates between $\rK(f,R(f))$ and $\z$.

Assume second that $f(\y)-\varepsilon\le s<t_{\max}(f)$.  Choose $t\in[f(\y),t_{\max}(f)]$ which is closest to $s$.  Then $|s-t|\le \varepsilon$.  Set $\w=(\y^\top,t)^\top$.  Note that $\w\in \rK(f,R(f))$.  Clearly,  $\|\z-\w\|_2\le \varepsilon$.  Hence $\dist_2(\z,\rK(f,R(f)))\le \varepsilon$.

Assume third that $s<f(\y)-\varepsilon$.  We now need to find a weakly separating $\phi_{\bc}$,  where $\bc\in\Q^{n+1}, \|\bc\|_2\ge 1$, between $\rK(f,R(f))$ and $\z$.  Assume that $f(\y)=f_{k}(\y)$ for some $k\in[N]$. 
As $f_{k}$ is a smooth convex function it follows that 
\begin{equation*}
f_{k}(\x)\ge f_{k}(\y)+\nabla f_{k}(\y)^\top(\x-\y) \textrm{ for all } \x\in\R^n.
\end{equation*}
Hence, 
\begin{equation*}
\begin{aligned}
&f(\x)\ge f_{k}(\y)+\nabla f_{k}(\y)^\top(\x-\y) \textrm{ for all } \x\in\R^n \Rightarrow\\
&-f(\x)+\nabla f_{k}(\y)^\top\x\le -f(\y)+\nabla f_{k}(\y)^\top \y \textrm{ for all } \x\in\R^n.
\end{aligned}
\end{equation*}
Assume that $(\x,t)\in\rK(f,R(f))$.  Then,
\begin{equation*}
\begin{aligned}
&-t+\nabla f_{k}(\y)^\top\x\le -f(\x)+\nabla f_{k}(\y)^\top\x\le
-f(\y)+\nabla f_{k}(\y)^\top \y<-s+\nabla f_{k}(\y)^\top\y-\varepsilon.
\end{aligned}
\end{equation*}
Hence $\tilde \phi=\phi_{(\nabla f_{k}(\y)^\top, -1)^\top}$ separates $\rK(f,R(f))$ and $\z$.
We now need to make a good rational approximation of $\nabla f_{j}(\y)$.
Clearly, if $\nabla f_{j}(\y)=0$  then $\tilde \phi$ has rational coefficients.
It is left to discuss the case where $\nabla f_{j}(\y)\ne 0$.  Use Lemma \ref{approxexlog} to deduce that we can find in a polynomial time
in $\<f>$ and in $|\log\varepsilon|$ a vector $\be\in\Q^{n}$ such that  
\begin{equation*}
\|\nabla f_{j}(\y)-\be\|_1\le \frac{\varepsilon}{2R(f)}.
\end{equation*}
Hence,
\begin{equation*}
-t+\be^\top \x\le -s+\be^\top \y +(\be-\nabla f_{j}(\y))^\top(\x-\y)-\varepsilon.
\end{equation*}
Recall that 
\begin{equation*}
|(\be-\nabla f_{j}(\y))^\top(\x-\y)|\le \|\be-\nabla f_{j}(\y)\|_1\|\x-\y\|_\infty\le 
\frac{\varepsilon}{2R(f)} (2R(f)).
\end{equation*}
Let $\bc=(\be^\top,-1)^\top$.  Then $\phi_{\bc}$ separates between $\rK(f,R(f))$ and $\z$.
\end{enumerate}
\end{proof}

\begin{proof}[Proof of Theorem~\ref{mainthm}]
Use  Proposition \ref{weakseprop} and \cite[Theorem 3.1]{GLS81} to deduce the ellipsoid algorithm for the minimum of the linear functional $\phi_{\be_{n+1}}$ on $\rK(f,R(f))$ finds $\z^\star=((\x^\star)^\top,t^\star)^\top \in\Q^{n+1}$ such that $\dist_2(\z^{\star},\rK(f,R(f)))\le \varepsilon$, and $t^\star=\phi_{\be_{n+1}}(\z^\star)\le f_{\min}+\varepsilon$.  Assume that $\dist_2(\z^{\star},\rK(f,R(f)))=\|\z^\star-\w\|_2$ for some $\w=(\x^\top, t)^\top\in \rK(f,R(f))$.  Hence $|t^{\star}-t\|\le \varepsilon$.    Thus
\begin{equation*}
  t^\star\ge t-\varepsilon\ge f(\x)-\varepsilon\ge f_{\min}-\varepsilon.
  \qedhere
\end{equation*}
\end{proof}
\section{Interior-point methods for finding $f_{\min}$}\label{sec:intmethod}
In this section we will show that one can apply interior-point methods (IPM) for $f_{\min}$ given by \eqref{e-geo}.   The basic case $N=1$ in \eqref{deff(x)fj} was done in \cite{BLNW20}, and we point out briefly how to use their results for $N>1$.
We will mainly use the references \cite{NN94, Ren01}, 
\subsection{Preliminaries}
We first recall some notations and definitions that we will use in this section.
Let $g\in\rC^3(\rB_2(\x,r))$.   Denote 
\begin{equation*}
g_{,i_1\ldots i_d}(\x)=\frac{\partial ^d}{\partial x_{i_1}\ldots\partial x_{i_d}}g(\x), \quad i_1,\ldots,i_d\in[n], d\in[3].
\end{equation*}
Recall the Taylor expansion of $g$ at $\x$ of order $3$ for $\bu\in\R^n$ with a small norm:
\begin{equation*}
\begin{aligned}
&g(\x+\bu)\approx g(\x)+\partial g(\x)+\frac{1}{2}\bu^\top \partial^2 g(\x)\bu+\frac{1}{6}\langle\partial ^3  g(\x), {\otimes^ 3}\bu\rangle,\\
&\partial g(\x)=(g_{,1}(\x),\ldots,g_{,n}(\x))^\top,  \quad \partial^2 g(\x)=[g_{,ij}(\x)],  i,j\in[n],\\ 
&\partial^3 g(\x)=[g_{,ijk}(\x)], i,j,k\in[n], \, \langle\partial^3 g(\x),\otimes^ 3\bu\rangle=\sum_{i,j,k\in[n]}g_{,ijk}(\x)u_i u_j u_k, 
\end{aligned}
\end{equation*}
where $\partial f, \partial^2 f, \partial ^3 f$ are called
the gradient, the Hessian, and the 3-mode symmetric partial derivative tensor of $f$. 

For a matrix $A\in\R^{n\times n}$ denote by  $A^\dagger\in\R^{n\times n}$ the Moore-Penrose inverse of $A$ \cite[\textsection 4.12]{Frib}. Recall that if $A$ is invertible then $A^\dagger=A^{-1}$.
A set $\rD\subset \R^n$ is called a domain if $\rD$ is an open connected set.
\begin{definition}\label{defconcconst}  
Assume that $g: \rD\to\R$ is a convex function 
 defined in a convex domain $\rD\subset \R^n$,  and that $g\in\rC^3(\rD)$.   The function $g$ is called {\em $a (>0)$-self-concordant}, or simply {\em self-concordant}, if the following inequality hold
\begin{equation}\label{defconcconst1} 
|\langle \partial^3g(\x),\otimes^3\bu\rangle|\le 2a^{-1/2} (\bu^\top \partial^2g(\bx)\bu)^{3/2}, 
\textrm{ for all }\x\in \rD,\bu\in\R^n.
\end{equation}
The function $g$ is called {\em standard self-concordant} if $a=1$,  and {\em strongly $a$-self-concordant} if $g(\x_m)\to\infty$ whenever a sequence $\{\x_m\}$ converges to the boundary of $\rD$.

The {\em complexity value} $\theta(g)\in[0,\infty]$ of a-self-concordant function $f$ in $\rD$, called a self-concordance parameter (\cite[Definition 2.3.1]{NN94}), is
\begin{equation}\label{defsconc0}
\begin{aligned}
&\theta(g)=\sup_{\x\in\rD}
\inf\{\lambda^2\in[0,\infty], |\partial g(\x)^\top \bu|^2
\le \lambda^2 a\big(\bu^\top\partial^2 g(\x)\bu\big),\forall \bu\in\R^n\}.
\end{aligned}
\end{equation}
 A strongly 1-self-concordant function with a finite $\theta(g)$ is called a {\em barrier} (function).   A {\em $\theta$-self-concordant} barrier is a barrier that satisfies $\theta(g)\le \theta$.
 
 Assume that $\rD$ is a bounded domain,  $\x\in\rD$  and $\rL$ is a line  through $\x$. 
 Denote by $d_{\max}(\x,\rL)\ge d_{min}(\x,\rL)$ the two distances from $\x$ to the end points of $\rL\cap\rD$.   Then, we define $\operatorname{sym}(\x,\rD)$ as the infimum of $\frac{d_{\min}(\x,L)}{ d_{\max}(\x,L)}$ for all lines $\rL$ through $\x$. 
\end{definition}

Observe that Renegar \cite{Ren01} deals only with strongly standard self-concordant functions.  
Recall that the complexity value $\theta(f)$, coined in \cite{Ren01}, is referred to as the {\em parameter} of barrier $f$ in \cite{NN94}, and is deined only for self-concordant barriers in \cite[\textsection 2.3.1] {NN94}.   

A simple implementation of Newton's method for minimizing a linear functional over a convex bounded domain $\rD$, using a given barrier function $\beta$, is  a "short-step" IPM that  follows the central path \cite[\textsection 2.4.2]{Ren01}.
The number of iterations required to approximate the minimum of a linear functional to within $\varepsilon$ accuracy, starting with an initial point $\x'$, is given by \cite[Theorem 2.4.1]{Ren01}:
\begin{equation}\label{Renthm}
O\big(\sqrt{\theta(\beta)}\log\big(\frac{\theta(\beta)}{\varepsilon \operatorname{sym}(\x',\rD)}\big)\big).
\end{equation}

Corollary 2.4 \cite{Fri23} states:
\begin{lemma}\label{charthetf}  Let $\rD\subset\R^n$ be a convex domain and assume that $g$ is a nonconstant $a$-self-concordant function on $\rD$.    Suppose furthermore that
$\partial g(\x)^\top \ker \partial^2 g(\x)=\0$ for each $\x\in\rD$.   Then
\begin{equation}\label{defconcconst12}
 \theta(g)=a^{-1}\sup_{\x\in \rD}\partial g(\x)^\top (\partial^2g)^\dagger(\x)\partial g(\x).
\end{equation}
\end{lemma}

Observe that if $g$ an $a$-self-concordant function on a convex domain $D$ then $a^{-1}g$ is a standard self-concordant function in $D$, and $\theta(a^{-1}g)=\theta(g)$.   The following proposition is a consequence of \cite[Proposition 2.3.1]{NN94} and Lemma \ref{charthetf}:
\begin{proposition}\label{propbar}  Let $\rD\subset\R^n$ be a convex domain and $k\in\N$.  Assume that $g_i$ is a standard self-concordant function in $\rD$ with $\theta(g_i)<\infty$ for $i\in[k]$.   Then
\begin{enumerate}[label=(\alph*)]
\item Let $k=1$ and $\bA:\R^n\to\R^m$ be an onto affine transformation.
Then $\tilde g_1:=g_1\circ \bA$ is a standard self-concordant function in the convex domain $\bA(D)$, and $\theta(\tilde g_1)\le \theta(g_1)$.
\item The function $g=\sum_{i=1}^k g_i$ a standard self-concordant function in $D$, and $\theta(g)\le \sum_{i=1}^k \theta(g_i)$.
\end{enumerate}
\end{proposition}

Recall that the functions $-\log x$,  $-\log x-\log(\log x-y)$ and $-\log (R^2-\|\x\|_2^2)$ are  self-concordant barriers in $\Rpp$, $\{(x,y)\in\R^2, x>0, \log x>y\}$, and $\{\x\in\R^n,\|\x\|_2<R\}$ respectively, that satisfy 
\begin{equation}\label{thetexam}
\theta(-\log x)=1,   \theta(-\log x-\log(\log x-y))\le 2,   \theta(-\log(R^2 -\|\x\|_2^2))=1.
\end{equation} 
See \cite[Ex. 1, Sec 2.3; Prop.  5.3.3; Prop. 5.4.2]{NN94}.
\subsection{IPM to compute $f_{\min}$}\label{subsec:ipm}
In this subsection we assume that the conditions of Lemma \ref{contlemma} hold. 
Denote
\begin{equation}\label{defPin}
\Delta^n=\{\x=(x_1,\ldots,x_n)^\top\in\R^n,  \sum_{i=1}^n x_i=1\}, \quad \Deltapp^n= \Delta^{n}\cap\Rpp^n.
\end{equation}
We view $\Delta_n$ as an affine $n-1$ dimensional space.   For the affine space 
\begin{equation}\label{defSigma}
\Sigma:=\R^{n+1}\times \Delta^{m_1}\times \cdots \times \Delta^{m_N}
\end{equation}
 we denote by $\|\cdot\|_p$ the restriction of the $\ell_p$ norm on $\R^{n+1+\sum_{j=1}^N m_j}$.  Then by $B'_p(\w,r)\subset \Sigma$  we denote the restriction of the ball $B_p(\w,r)\subset \R^{n+1+\sum_{j=1}^N m_j}$ to $\Sigma$.
 
We now follow the arguments of \cite{BLNW20} to replace the set $K(f,R)\subset \R^{n+1}$, defined in Lemma~\ref{propKfR}, by a new set $\hat K(f,R)\subset \R^{M}$, defined in \eqref{defhatK(f,R)},  with a barrier $\beta$:
\begin{proposition}\label{defhatKf}  Let 
\begin{equation}\label{Ajdef}
\begin{aligned}
&\cA_j=\{\ba_{1j},\ldots,\ba_{m_jj}\}\subset \R^{n},\quad f_{\ba_{ij}} >0 \,\mathrm{for}\, i\in[m_j],
\quad f_j(\x)=\log\sum_{i=1}^{m_{ij}} f_{\ba_{ij}}\exp(\ba_{ij}^\top\x), \quad j\in[N],\\
&\mu=\min_{j\in[N]}\min_{i\in[m_j]}\frac{f_{\ba_{ij}}}{\sum_{i=1}^{m_j} f_{\ba_{ij}}}\enspace .
\end{aligned}
\end{equation}
Let the definitions of \Cref{propKfR} hold.   Suppose furthermore that 
\begin{equation}\label{condR}
R\ge\max\Big(\frac{f(\0)- t_{\min}(f)}{\nu(\cA)}, 1 +\log 4\Big).
\end{equation}
Denote by $\z_j=(z_{1j},\ldots,z_{m_jj})^\top$ the variables of $\Delta^{m_j}$ for $j\in[N]$.
Define
\begin{equation}\label{defhatK(f,R)}
\begin{aligned}
\hat K(f,R)=\{(\x,t,\z_1,\ldots,\z_N)\in K(f,R)\times \Deltapp^{m_1}\times \cdots \times \Deltapp^{m_N},\;
f_{\ba_{ij}}e^{\ba_{ij}^\top\x}\le z_{ij}e^t, i\in[m_j], j\in[N]\}.
\end{aligned}
\end{equation}
Then
\begin{enumerate}[label=(\alph*)]
\item The set $\hat K(f,R)$ is a compact convex set.
\item The projection of $\hat K(f,R)$ on the first $n+1$ coordinates is $K(f,R)$.
\item $f_{\min}=\min\{t,  (\x,t,\z_1,\ldots,\z_N)\in \hat K(f,R)\}$.
\item Let 
\begin{equation*}
\begin{aligned}
  \z_j(f)=\frac{1}{\sum_{i=1}^{m_j} f_{\ba_{ij}}}(f_{\ba_{1j}},\ldots,f_{\ba_{m_jj}})^\top, \quad j\in[N],\qquad
  \bw=(\0,\bar t(f,R), \z_1(f),\ldots,\z_N(f))^\top.
\end{aligned}
\end{equation*}
Then
\begin{equation}\label{contcondhatK(f,R(f))}
\begin{aligned}
B_{2}(\bw,\mu/2)\subset \hat K(f,R)\subset B_{2}\Big(\bw,\sqrt{n+1+\sum_{j=1}^N m_j}R_{\max}\Big).
\end{aligned}
\end{equation}
\item The function 
\begin{equation}\label{barhatK(f)}
\begin{aligned}
  \beta=-\sum_{j=1}^N\sum_{i=1}^{m_j}\big(\log z_{ij}+\log(\log z_{ij}-\ba_{ij}^\top \x +t -\log f_{\ba_{ij}})\big)
  -
  \sum_{k=1}^n\log(R^2 -x_k^2) -\log(t_{\max}(f,R)-t)
\end{aligned}
\end{equation}
is a $\big(1+n+2\sum_{j=1}^N m_j\big)$-self-concordant barrier of the interior of $\hat K(f,R)$.
\end{enumerate}
\end{proposition}
\begin{proof}
(a)  Clearly, $\hat K(f,R)$ is a compact set.
It is left to show the convexity of $\hat K(f,R)$.  Assume that $(\x^{(l)},t^{(l)},\z_1^{(l)},\ldots,\z_N^{(l)})^\top \in\hat K(f,R)$ for $l\in[2]$.  Suppose that the average of these points is $(\x,t,\z_1,\ldots,\z_N)^\top$.   As $K(f,R)$ is convex we obtain that $(\x,t)\in K(f,R)$.
It is left to show that  
\begin{equation*}
\log f_{\ba_{ij}}+\ba_{ij}^\top \x\le \log z_{ij} + t, \quad i\in[m_j], j\in[N].
\end{equation*}
This follows from the concavity of $\log z$.

\noindent
(b)  Assume that $(\x,t)\in K(f,R)$.  Hence $e^{f_{j}(\x)}\le e^{t}$ for $j\in[N]$.  Let 
\begin{equation*}
\bu(t)=(\x,t,\z_1(t),\ldots,\z_N(t))^\top, 
\z_{ij}(t)=e^{-f_{j}(\x)}f_{\ba_{1j}}e^{\ba_{ij}^\top \x}, i\in[m_j],j\in[N].
\end{equation*} 
As $\sum_{i=1}^{m_j}z_{ij}(t)=1$ we deduce that $\bu(t)\in \hat K(f,R)$.  Hence, the projection of $\hat K(f,R)$ on the first $n+1$ coordinates is $K(f,R)$.

\noindent
(c) The inequality \eqref{condR} yields that $f$ coercive.   Hence,  \eqref{minfchar} holds.  Therefore, (b) implies (c).

\noindent
(d) We claim
\begin{equation}\label{contcondhatK(f,R(f))inf}
\begin{aligned}
B_{\infty}(\bw,\mu/2)\subset \hat K(f,R)\subset B_{\infty}((\bw,R_{\max}).
\end{aligned}
\end{equation}
Observe that $R_{\max}\ge R\ge 1 +\log 4>1 \ge \mu$.  
Recall the condition \eqref{contcondK(f,R)}.  As $z_{ij}\in[0,1]$ we deduce the right hand side of \eqref{contcondhatK(f,R(f))inf}.   We now show the left hand side of  \eqref{contcondhatK(f,R(f))inf}.

The left-hand side of the condition \eqref{contcondK(f,R)} yields 
\begin{equation*}
B_{\infty}((\0, \bar t(f,R)), \mu/2)\subset B_{\infty}((\0, \bar t(f,R)), R/2)\subset K(f,R).
\end{equation*}
It is left to show the inequality
\begin{equation*}
f_{\ba_{ij}}e^{\ba_{ij}^\top\x}\le z_{ij}e^t \textrm{ for } |z_{ij}-z_{ij}(f)| , |t-\bar t(f,R)|\le \mu/2, \|\x\|_{\infty}\le \mu/2.
\end{equation*}
From the definition of $z_{ij}(f)$ and $\mu$ it follows that $z_{ij}\ge z_{ij}(f)/2$.  Clearly, $t\ge \bar t(f,R)-\mu/2$.    Thus, it is enough to show
\begin{equation*}
\begin{aligned}+
  f_{\ba_{ij}}e^{\ba_{ij}^\top\x}\le \frac{f_{\ba_{a_{ij}}}e^{\bar t(f,R)-\mu/2}}{2\sum_{i=1}^{m_j}f_{\ba_{ij}}}=\frac{1}{2}e^{-f_{j}(\0)}f_{\ba_{a_{ij}}}e^{\bar t(f,R)-\mu/2}\iff
\log 2+f_{j}(\0)+\ba_{ij}^\top \x\le \bar t(f,R)-\mu/2.
\end{aligned}
\end{equation*}
Recall that $f_{\cA_j}(\0)\le f(\0)$.  The definition of $a(f)$ and $\bar t(f,R)$  in Lemma~\ref{propKfR} yield
\begin{equation*}
\begin{aligned}
\log 2+f(\0)+\ba_{ij}^\top \x\le \log2 +f(\0)+a(f)\mu/2,\qquad 
\bar t(f,R)-\mu/2 =f(\0)+\frac{(a(f)+1)R}{2}-\mu/2.
\end{aligned}
\end{equation*}
As $1\le \mu$ and $R\ge 1+\log 4$ we deduce the left hand side of \eqref{contcondhatK(f,R(f))inf}.

The containment of the balls in \eqref{normin} and \eqref{contcondhatK(f,R(f))inf} yield
\eqref{contcondhatK(f,R(f))}.

\noindent
(e) From the definition of $\hat K(f,R)$ it follows that the function $\beta$ is a convex $C^{\infty}$ function in the interior int$(\hat K(f,R))$ , which blows to infinity when $\bu$ approaches the boundary of $(\hat K(f,R))$.  The equalities and inequality in \eqref{thetexam} shows that $\beta$ is $\big(1+n+2\sum_{j=1}^N m_j\big)$-self-concordant barrier of the interior of $\hat K(f,R)$.
\end{proof}

Thus we can apply the above Proposition under the assumptions of Lemma~\eqref{nulowest}.
\begin{corollary}\label{IPMfmin} Suppose that the assumptions of Lemma \ref{nulowest} hold.   Set
\begin{equation*}
\begin{aligned}
R&=(\lceil f(\0)\rceil- \lfloor t_{\min}(f)\rfloor)\lceil\sqrt{\prod_{k=1}^n (\|\bb_k\|^2_2+1)}\,\rceil+1+\lceil \log 4\rceil,\\
t_{\max}(f,R)&=\lceil f(\0)\rceil +(a(f)+1)R,  \\
\bar t(f,R)&=\lceil f(\0)\rceil+\frac{(a(f)+1)R}{2},\\
R_{\max} &=(a(f)+1)R+\lceil f(\0)\rceil-\lfloor t_{\min}(f)\rfloor,\\
K(f,R)&=\{(\x,t), \x\in\R^n, t\in\R, f(\x)\le t, \|\x\|_{\infty}\le R, t\le t_{\max}(f,R)\}.
\end{aligned}
\end{equation*}
Assume that $\hat K(f,R)$ and $\bw$ are given as in Proposition \ref{defhatKf}.
Consider the minimum problem in part (c) of Proposition \ref{defhatKf}.  
The short-step interior-point algorithm with the barrier $\beta$ starting at the point $\bw$ finds the value $f_{\min}$ within precision $\varepsilon>0$ in  
\begin{equation}\label{ipmotmat1}
O\big(\sqrt{1+n+2\sum_{j=1}^N m_j}\log\frac{2\big(1+n+2\sum_{j=1}^N m_j\big)R_{\max}}{\varepsilon \mu}\big)
\end{equation}
iterations.
\end{corollary}
The complexity bound in  Corollary~\ref{IPMfmin} only concerns the number of iterations. To obtain a polynomial-time complexity in the Turing model, one must also show that at each iteration, the numbers in the implementation of the "short-step" IPM can be rounded to maintain short rational representations. We leave this investigation for further work.\todo[inline]{SG:
Deleted ``We believe that this can be done as in \cite{deklerk_vallentin}'' as this relies on the dual central path}
\section{Polynomial time approximation of $\rho(\bF)$}\label{sec:minconfun}
\subsection{Conversion to a minimax problem}\label{subsec:conver}
Let $\cM_{d-1,n}\subset \Zp^{n}$ be the set of all vectors $\ba=(a_1,\ldots,a_{n})^\top$ with nonnegative integer coordinates whose sum is $d-1$.  Recall that $|\cM_{d-1,n}|={n+d-2\choose d-1}$. 
To each $\ba\in\cM_{d-1,n}$ we associate the monomials $\z^{\ba}=z_1^{a_1}\cdots z_{n}^{a_{n}}$ of degree $d-1$.
 Let $\bF:\Rp^{n}\to \Rp^{n}$ be a homogeneous polynomial map of degree $d-1\ge 1$ of the form
\begin{equation}\label{defbFi}
F_i(\z)=\sum_{\bb\in\cB_i} f_{\bb,i}z^{\bb}, \quad\cB_i\subseteq \cM_{d-1,n}, f_{\bb,i}>0, \; i\in[n], \;\bb \in \cB_i
\end{equation}
in which for each $i\in [n]$, $\cB_i$ is a non-empty set of exponents of monomials
effectively appearing in $F_i$.

We define for $\bx\in\R^n$, 
\begin{align}
f_i (\bx):= \log F_i(\exp(\bx))-(d-1)x_i, \; i\in [n], \qquad f:=\max_{i\in[n]} f_i \enspace ,\label{e-def-f}
\end{align}
and we set
\[
\cA_i := \{ \bb - (d-1) \be_i\mid \bb \in \cB_i \}, \;i\in [n], \qquad \cA:= \cup_{i\in [n]} \cA_i
\]
where $\be_i$ denotes the $i$th vector of the canonical basis of $\R^n$.
We shall refer to $f$ as the {\em Collatz-Wielandt function},
for it follows from the Collatz-Wielandt minimax formula (first equality in~\eqref{minmaxchar}) that if $\bF$ is weakly irreducible, $\log\rho(\bF)= \min_{\bx\in \R^n} f(\bx)$. Since $f(\bx + \lambda \1) = f(\bx)$ holds for all $\lambda\in\R$, the map $f$ cannot be coercive
on $\R^n$. However, we shall show that it is coercive on the subspace
$E_n:=\{\bx \in \R^n\mid x_n =0\}$.
\begin{theorem}\label{Fcoercthm1} Let $\bF=(F_1,\ldots,F_{n}):\Rpp^{n}\to\Rpp^{n}$ be a weakly irreducible $(d-1)$-homogeneous map.  
  Then the Collatz-Wielandt function $f$ is coercive on the subspace $\mathbf{E}_n$, and
  \[ \nu(\cA, \mathbf{E}_n)\ge \big(2(d-1)^2+1\big)^{-(n-1)/2} \enspace.
  \]
  Moreover,
  an $\varepsilon$-approximation of $\rho(\bF)$ can be obtained in polynomial time in $\<F>$ and $|\log\varepsilon|$.
\end{theorem}


\begin{proof}
  Let $\|\bx\|_H:=\max_{i\in [n]} x_i -\min_{i\in [n]} x_i$ be the Hilbert's semi-norm on $\R^n$.  
  We write a general element of $\mathbf{E}_n$ as $\bx=(\tilde{\bx},0)$ with $\tilde{\bx}\in \R^{n-1}$,
  and set $\tilde{f}(\tilde{\bx}):=f((\bx,0))$.  Observe that for $\bx\in \mathbf{E}_n$ we have that
  $\max_{i\in[n]}x_i\ge 0$ and $\min_{i\in[n]}x_i\le 0$.   Hence $\|\bx\|_\infty\le \|\bx\|_H\le 2\|\x\|_\infty$ for $\x\in \mathbf{E}_n$.  Thus, the restriction of Hilbert semi-norm to $\mathbf{E}_n$ is actually a norm.
  We first show that $f$ is coercive on $\mathbf{E}_n$, or equivalently,
  that $\tilde{f}$ is coercive on $\R^{n-1}$. Since $\bF$ is weakly irreducible,
  by~\cite[Th.~4]{GG04}, every sub-level set $\{\bx \in \R^n\mid f(\bx) \leq \lambda \}$ is bounded
  in Hilbert's semi-norm,
  and this entails that every sub-level set of $\tilde{f}$ is compact.
  Then, the condition of Part (b) of Lemma \ref{coerclemma} is satisfied,
  showing that $\tilde{f}$ is coercive. We now apply \Cref{nulowest} to the map $\tilde{f}$.
  The exponents of the monomials appearing in $\tilde{f}$ are of the form $\tilde{\bb}-(d-1)\be_i$, for $i\in [n-1]$ and $\bb\in \cB_i$ or of the form $\tilde{\bb}$ with $\bb\in \cB_n$, where
  $\tilde{\bb}\in \R^{n-1}$ denotes the monomial gotten by deleting the $n$ coordinate
  of $\bb$. Observe that for all $\bb\in \cup_{i\in[n]} \cB_i$, $\|\bb\|_2\leq \|\bb\|_1 = d-1$.
  Hence, $\|\tilde{\bb}-(d-1)\be_i\|_2=
  \sqrt{\sum_{j}\bb_j^2 -2(d-1) \bb_i + (d-1)^2}\leq
  \sqrt{\sum_{j}\bb_j^2 + (d-1)^2}\leq \sqrt{2} (d-1)$
for $i\in [n-1]$ and $\bb\in \cB_i$ whereas
  $\|\tilde{\bb}\|_2\leq d-1$ for $\bb\in \cB_n$.
  Hence, the inequality \eqref{nulowest1}  yields $\nu(\cA(\bF))\ge \big(2(d-1)^2+1\big)^{-(n-1)/2}$.   Theorem \ref{mainthm} yields that a $\varepsilon$ approximation of $\log\rho(\bF)$
  can be computed in polynomial time.

  We now show how to approximate $\rho(F)$. Using the minimax characterization of the spectral radius~\eqref{minmaxchar} and specializing $\bx$ to the unit vector $\1$ in~\eqref{minmaxchar},
  and noting that $|\cB_i|\leq |\cM_{d-1,n}| ={n+d-2\choose d-1}$, we get
\[
\log \rho(\bF)\leq a:=\log {n+d-2\choose d-1} + \max_{\bb,j} \log f_{\bb,j} \enspace .
\]
Hence, using the fact that the exponential map has Lipschitz constant $\exp(a)$ on the interval $(-\infty,a]$, we deduce that
if $\alpha$ is an $\eta$-approximation of $\log \rho(\bF)$, then $\exp(\alpha)$ is an $\varepsilon$-approximation of $\rho(\bF)$, with
\[\varepsilon= \eta {n+d-2\choose d-1} \max_{\bb,j} f_{\bb,j}
\enspace .
\]
Thus, to
end up with an $\varepsilon$-approximation of $\rho(\bF)$, it suffices
to approximate $\log \rho(\bF)$ with a precision $\eta$ of bit-size
polynomially bounded in the input size.%
\end{proof}
\subsection{Entropic characterization of $\rho(\bF)$}\label{subsec:entrchar}
Assume that $\bF:\R^{n} \to \R^{n}$ is as in~\eqref{e-hommap}.

We now recall the entropic characterization of $\rho(\bF)$ \cite[Section 6.3]{FG20}.
A tensor $\mu=[\mu_{i_1,\ldots,i_d}]\in \R_{ps,+}^{n^{\times d}}$ is called
an {\em occupation measure} if
\begin{eqnarray}
 &&\sum_{i_1\in [n],\dots, i_d\in [n]}\mu_{i_1,i_2,\dots, i_d}=1,\notag\\
 &&\sum_{i_2,\dots,i_d\in [n]} 
 \mu_{j,i_2,\dots,i_d} = \displaystyle\sum_{i_1,i_3,\dots,i_d\in [n]
}
 \mu_{i_1,j,i_3,\dots,i_d} 
 \qquad \forall j\in [n]\enspace .\label{tenocmeas}
\end{eqnarray}
Note that in view of the partial symmetry of $\mu$ the condition \eqref{tenocmeas}
is equivalent to
\[\sum_{i_2,\dots,i_d\in[n]} 
 \mu_{j,i_2,\dots,i_d} = \displaystyle\sum_{i_k=j,i_1,\dots,i_{k-1},i_{k+1},\dots,i_d\in [n] 
}
 \mu_{i_1,\dots,i_d} 
 \qquad \forall j\in [n]\enspace ,\]
 for $k\in\{2,\ldots,n\}$.\todo[inline]{SF: See part (c) on page 7}
We denote by $\fO(n^{\times (d-1)})\subset \R_{ps,+}^{n^{\times d}}$ the set of occupation measures.  For $\cT\in\R_{ps,+}^{n^{\times d}}$ we denote by $\fO(n^{\times (d-1)},\supp \cT)\subseteq \fO(n^{\times (d-1)})$ the set of occupation measures whose support is contained in $\supp \cT$.

Assume that $\bF$ is weakly irreducible.
Then there exists a unique positive  eigenvector $\bu>\0$  (up to rescaling) such that
\begin{equation}\label{poseigvec}
\bF(\bu)=\rho(\bF)\bu^{\circ(d-1)}, \quad \bu>\0.
\end{equation}
Let $D(\bF)(\x)=[\frac{\partial F_i(\x)}{\partial x_j}]$ be the Jacobian matrix of $\bF$.
Set
\begin{equation}\label{defATu} 
A:=\diag(\bu)^{-(d-2)}D(\bF)(\bu), \quad \diag(\bu)=\diag(u_1 ,\cdots,u_{n}).
\end{equation}
Recall that $A$ is a nonnegative irreducible matrix that satisfies \cite[(3.18)]{FG20}:
\begin{equation}\label{defpropA}
A\bu=(d-1)\rho(\bF)\bu, \quad A^\top\w=(d-1)\rho(\bF)\w,\quad\w^\top \uu=1.
\end{equation}

Theorem 19 in \cite{FG20} states:
\begin{theorem}[Entropic characterization of the spectral radius]\label{entropcharsrten}
The spectral radius of  a nonzero nonnegative $\bF=(F_1,\ldots,F_{n})^\top$,  has the following characterization 
\begin{equation}\label{entropcharsrten1}
\begin{aligned}
\log \rho(\bF) = 
\max_{\mu\in \fO(n^{\times(d-1)})}
\sum_{i_1,\dots,i_d\in [n]}
\mu_{i_1,\dots, i_d} \log \Big(\frac{(\sum_{k_2,\dots, k_d}\mu_{i_1,k_2,\dots, k_d})f_{i_1,\dots, i_d}}{\mu_{i_1,\dots, i_d}}\Big) \enspace .
\end{aligned}
\end{equation}
Assume that $\bF$ is weakly irreducible. Let $\uu>0$ be the unique positive eigenvector $\uu$ satisfying \eqref{poseigvec} and let $\w>\0$ be defined as in \eqref{defpropA}.  Let $\mu=[\mu_{i_1,\ldots,i_d}]\in \R_{ps,+}^{n^{\times d}}$ be the tensor given by
\begin{equation}\label{mumaxten}
\mu_{i_1,\ldots,i_d}=\frac{1}{\rho(\bF)} w_{i_1}u_{i_1}^{-(d-2)}f_{i_1,\ldots,i_d}u_{i_2}\cdots u_{i_d} \textrm{ for } i_1,\ldots,i_d\in[n].
\end{equation}
Then $\mu$ is an occupation measure whose support is $\supp \bF$. Furthermore,
\begin{equation}\label{logeqtenT}
\log\rho(\bF)=\sum_{i_1,\ldots,i_d\in [n]} \mu_{i_1,\dots, i_d} \log \Big(\frac{(\sum_{k_2,\dots, k_d}\mu_{i_1,k_2\dots ,k_d})f_{i_1,\dots, i_d}}{\mu_{i_1,\dots ,i_d}}\Big) \enspace .
\end{equation}
\end{theorem}

\begin{corollary}\label{corpoltrho(F)}{(of Theorem \ref{Fcoercthm1})}
Let $\bF:\R^{n}\to \R^{n}$ be a homogeneous polynomial map of degree $d-1\ge 1$ with nonnegative coefficients.  Denote by $\cF=[f_{i_1,\ldots,i_d}]\in \R_{ps,+}^{n^{\times d}}$ the partially symmetric tensor corresponding to $\bF$.  Assume that $\bF$ is weakly irreducible.
Then for a given $\varepsilon\in\Qpp$, the maximum given in \eqref{entropcharsrten1} can be approximated within $\varepsilon$ precision in polynomial time. 
\end{corollary}
\subsection{A weakly irreducible map $\bF$ with an eigenvector with doubly exponential coordinates}\label{subsec:expswit}

  Let $W$ be a positive rational number, and $d\geq 3$.
We define the homogeneous map of degree $d-1$, $\bF: \Rp^{2n}\to \Rp^{2n}$, 
\begin{equation}\label{Fexamp-new}
\begin{aligned}
  &F_i(\z)=W z_1^{d-2} z_{i+1} \textrm{ for } i=1,\ldots,n,\\
  &F_{n+i}(\z)=W^{-1} z_{n+1}^{d-2} z_{n+1+i} \textrm{ for } i=1,\ldots,n-1,\\
&   F_{2n}(\bz) =W^{-1} z_1z_{n+1}^{d-2} \enspace .
\end{aligned}
\end{equation}
We set $\bard:=d-1$.
\begin{figure}
  \begin{center}

    \begin{tikzpicture}[->, >=stealth, node distance=1.2cm, every node/.style={circle, draw, minimum size=8mm, inner sep=0pt}]
    \foreach \i in {1,...,5} {
        \node (T\i) at (\i*1.5, 2) {\small\i,$+\ell$};
    }
    
    \foreach \i in {1,...,5} {
        \pgfmathtruncatemacro{\j}{6-\i}
        \pgfmathtruncatemacro{\k}{5+\i}
        \node (B\i) at (\j*1.5, 0) {\small\k,$-\ell$};
    }
    
    \foreach \i in {1,...,4} {
      \draw (T\i) -- (T\the\numexpr\i+1);
    }
        \foreach \i in {1,...,4} {
          \draw (B\i) -- (B\the\numexpr\i+1);
    }

        \draw (T5) -- (B1); 
        \draw (B5) -- (T1); 
    
    \foreach \i in {2,...,5} {
        \draw[bend left=-20-\i*5] (T\i) to (T1);
    }
    \foreach \i in {2,...,5} {
        \draw[bend left=-20-\i*5] (B\i) to (B1);
    }
    \end{tikzpicture}
  \end{center}
  \caption{The Markov chain with rewards associated to the counter-example~\eqref{Fexamp-new}, with $n=5$.}\label{fig-mc}
 \end{figure}

\begin{lemma}\label{example}
  Suppose that $d\geq 3$. Then, 
  the map $\bF$ defined by~\eqref{Fexamp-new} is weakly irreducible.
  It has a spectral radius equal to $1$,
a unique positive eigenvector $\bu$ such that $u_1=1$, and
$\bU=(U_1,\ldots,U_{2n})^\top:=\log\bu $ satisfies:
\begin{align}
  \label{e-def-u}
U_i =-\bard \frac{\bard^{i-1}-1}{\bard -1} \log W, i=1,\dots,n+1, \;
U_{n+i} = -\bard \frac{\bard^n - \bard^{i-1}}{\bard -1}\log W , i=2,\dots,n \enspace .
\end{align}
\end{lemma}
\begin{proof}
  First, observe that the digraph $\overset{\rightarrow}G(\bF)$ contains a cycle with $2n$ vertices: $i\mapsto i+1$ for $i\in[2n]$, with the convention $2n+1\equiv 1$.
This implies that $\bF$ is weakly irreducible, and therefore, 
  $\bF$ has a unique positive eigenvector up to a scalar factor.
  Finally, a straightforward computation shows that by setting $\bu=\exp(\bU)$, where $\bU$ is defined
  by~\eqref{e-def-u}, we obtain $F(\bu)=\bu^{\circ(d-1)}$.
\end{proof}
\begin{proof}[Proof of \Cref{th-cex}]
Since $u_n/u_1 $ is doubly exponential in the input size, we deduce
\Cref{th-cex}, showing that there is no polynomial algorithm providing an $O(1)$ approximation
of the eigenvector $\bu$.
\end{proof}
\begin{remark}\label{rk-proba}
  The counter-example~\eqref{Fexamp-new} becomes intuitive through the following probabilistic interpretation. Consider the Markov chain depicted in Figure~\ref{fig-mc},
  with $2n$ states. We still use the convention $2n+1\equiv 1$. For all $1\leq i\leq 2n$, the probability of moving from state $i$ to state $i+1$ is $1/\bar{d}$. Moreover,
  for all $i=1,\dots,n$, the probability of moving from state $i$ to state $1$ is $1-1/\bar{d}$. Similarly, the probability of moving from state $n+i$ to state $n+1$ is $1-1/\bar{d}$. Each visit to a state $i=1,\dots,n$ generates an income of $\ell\coloneqq \log W$, whereas each visit to a state $n+i$ generates an income of $-\ell$. Let $P\in \R^{2n\times 2n}$ represent the transition matrix of this chain, and let $r=\ell (1,\dots,1,-1,\dots,-1)^\top \in \R^{2n}$ denote the vector of incomes. The eigenvector $F(\bu)=\bu^{\circ(d-1)}$ can then be expressed as $\bU=r+P\bU$.
  Then, if $X_0,X_1,\dots$ if a trajectory of the Markov chain with transition $P$,
  and if $\tau_{ji}$ is the first hitting time of state $i$ starting from
  state $j$, we have $U_j-U_i = \mathbb{E}(r_{X_0}+ \dots + r_{X_{\tau_{ji}-1}}| X_0=j)$,
  indicating that $U_j-U_i$ is the expected sum of incomes accumulated from state $j$ until reaching state $i$.
  Observing the graph in~\Cref{fig-mc}, we note that when $i\leq n$, $U_1-U_i = \ell\times \mathbb{E}\tau_{1i}$. To reach state $i$ from state $1$, one must traverse $i-1$ consecutive  forward arcs $j\mapsto j+1$, with $j=1,\dots,i-1$. This implies that the expected
  hitting time $\mathbb{E}\tau_{1i}$ is exponential in $i$, explaining the exponential dependence of $u_i$ on $i$ in~\eqref{e-def-u}.
\end{remark}
Consider now the modified map $\tilde{\bF}:\R^{2n}\to \R^{2n}$ such that
\[
\tilde F_1(\bz) = F_1(z) + z_{n+1}^{d-1}, \qquad \tilde F_i(\bz)=F_i(\bz), 2\leq i\leq 2n \enspace .
\]
\begin{lemma}\label{lemma-tight}  For a rational number $W>1$
  We have
  \[
  1< \rho(\tilde{\bF}) \leq
  1+W^{-{\bard}^{n+1}} \enspace .
  \]
\end{lemma}
\begin{proof} For $W\in\Q_{++}$, the characterization $\rho(\tilde{\bF})$ given by \Cref{minmaxchar} for the eigenvector $\bu$ in Lemma \ref{example} yields the inequalities
\begin{equation*}
1=\min_{i\in[2n]}\frac{\tilde F_{i}(\bu)}{u_i^{(d-1)}}\le \rho(\tilde{\bF})\le\max_{i\in[2n]}\frac{\tilde{F}_{i}(\bu)}{u_i^{(d-1)}}=F_1(\bu)=1+u_{n+1}^{d-1}=1+W^{-\bard^2(\bard^n-1)/(\bard -1)}
\end{equation*}
Observe that $(\bard^n-1)/(\bard -1)\ge \bard^{n-1}$.   Hence, for $W>1$ we get the inequality $ \rho(\tilde{\bF}) \leq 1+W^{-{\bard}^{n+1}}$.
Let $\tilde {\bF}(\bv)=\rho(\tilde{\bF})\bv^{\circ(d-1)}$.

To show the inequality $\rho(\tilde{\bF})>1$ we need to introduce the following maps.
Define for $\z\in\R_{++}^{2n}$ the following selfmaps of $\R_{++}^{2n}$:
\begin{equation*}
\begin{aligned}
&\bG(\z)=(G_1(\z),\ldots,G_{2n}(\z))=(F_1^{1/(d-1)}(\z),\cdots,F_{2n}^{1/(d-1)}(\z)), \\
&\tilde{\bG}(\z)=(\tilde G_1(\z),\ldots,\tilde G_{2n}(\z))=
((\tilde F_1)^{1/(d-1)}(\z),\cdots,\tilde F_{2n}^{1/(d-1)}(\z)).
\end{aligned}
\end{equation*}
We define the spectral radii of $\rho(\bG), \rho(\tilde{\bG})$ to be equal to $\rho^{1/(d-1)}(\bF), \rho^{1/(d-1)}(\tilde{\bF})$ respectively.  Clearly,  the spectral radii of $\bG$
has the characterization
\begin{equation*}
\rho(\bG)=\min_{\z\in \Rpp^n} \max_{i\in[n]} \frac{G_i(\z)}{z_i}=\max_{\z\in\Rp^n\setminus\{\0\}}\min_{i\in[n],z_i>0} \frac{G_i(\z)}{z_i} \enspace .
\end{equation*}
 Similarly, $\rho(\tilde \bG)$ have the above  characterizations.   Clearly, the eigenvectors corresponding to $\rho(\bG)$ and $\rho(\tilde{\bG})$ are $\bu$ and $\bv$ respectively.

 Denote by $\bG^k(\z)=\underbrace{\bG(\bG(\cdots(\bG(\z))\cdots)}_k$ and $\tilde{\bG}^k(\z)=\underbrace{\tilde{\bG}(\tilde{\bG}(\cdots(\tilde{\bG}(\z))\cdots)}_k$ the  $k$-iterates of $\bG$ and $\tilde{\bG}$ respectively.
From the arguments in \cite{GG04} it follows that $\tilde{\bG}^k$ has unique positive eigenvector $\bv$ with the eigenvalue (spectral radius) $\rho(\tilde{\bG}^k)=\rho^{k/(d-1)}(\tilde{\bG})$.    Furthermore,  $\rho(\tilde{\bG}^k)$ has similar characterization to $\rho(\bG)$.

  We have $\bu=\bG(\bu)\leq \tilde{\bG}(\bu)$, and $u_1 <\tilde G_1(\bu)$. Using the weak irreducibility of $\tilde \bF$, we see that the strict inequality propagates when repeatedly applying $\tilde{\bG}$, and
  deduce that $\bu \ll \tilde{\bG}^{2n} (\bu)$ in which $\ll$ means that the strict inequality holds
  on every coordinate.  Hence, there is a constant $\lambda>1$ such that $\lambda \bu \leq \tilde{\bG}^{2n} (\bu)$. Using the maximin characterization of the spectral
  radius (last expression in the above characterization of $\rho(\bG)$, we deduce that $\rho(\tilde{\bG}^{2n})\geq \lambda$,
  and so, $\rho(\tilde{\bG})\geq \lambda^{\frac{1}{2n}}>1\Rightarrow \rho(\tilde{\bF})>1$. 
\end{proof}
We next make use~\Cref{lemma-tight} to deduce that the Phase I program must be solved
with a precision having an exponential number of bits to decide the feasibility
of a geometric program.
\begin{proof}[Proof of~\Cref{th-phase1}]
  We consider the following feasibility problem, where $\tilde{\bF}$ is defined as above:
  \begin{align}
\text{does there exist a vector } \bv\in \Rpp^{2n} \text{ such that }
  \tilde{\bF}(\bv) \leq \bv^{\circ(d-1)} \enspace .\label{e-feasible}
  \end{align}
  By the minimax characterization of the spectral radius (see~\Cref{minmaxchar}),
  if~\eqref{e-feasible} was feasible, we would have $\rho(\tilde{bF})\leq 1$, contradicting~\Cref{lemma-tight}. Making the change of variable $\bV=\log\bv$, we rewrite~\eqref{e-feasible}
  as a feasibility problem of the form~\eqref{e-feasibility}, with $2n$ inequalities
  \[ f_i(\bV):= \log \tilde{F}_i(\exp(\bV))-(d-1)V_i \leq 0, i\in [2n] \enspace .\]
  The Phase I program reads
  \[
  \min t, t \geq f_i(\bV), i\in [2n], (\bV,t)\in \R^{2n}\times \R \enspace .
  \]
  Setting $\mu:=\exp(t)$, we see that every feasible solution
  $(\bv,t)$ of the Phase I program satisfies $\tilde{F}_i(\bv) \leq \mu \bv^{\circ d-1}$.
  Using again the minimax characterization of the spectral radius, we deduce
  from \Cref{lemma-tight}
  that the value of the Phase I program is $\log \rho(\tilde{\bF})\leq \log (1+ W^{-\bard^{n+1}})
  \leq W^{-\bard^{n+1}}$. Specializing $W=2$ and $d=3$, and noting then that $\<f>=\Theta(n)$,
  we see that the value of the Phase I program
  is in $2^{2^{-\Omega(\<f>)}}$, as claimed in the theorem.

  Consider now the same feasibility problem, with $\bF$ instead of $\tilde{\bF}$.
  By~\Cref{example}, setting $\bv=\bu$ yields a solution, and the value
  of the Phase I problem is zero.
  \end{proof}
\section{Polynomial time approximation of log-eigenvector of strongly irreducible $\bF$}\label{sec:aplogeig}
Recall that a $(d-1)$-homogeneous map is strongly irreducible if $D(\bF)(\1)$ has positive off-diagonal elements.  The aim of this section to prove the following theorem:
\begin{theorem}\label{sirFthm1}
  Assume that $\bF: \R_{++}^{n}\to \R_{++}^{n}$ is a strongly irreducible homogeneous polynomial map with coefficients
$f_{\bb,j}\in \Qpp$ for $\bb\in \cB_j, j\in[n]$.   Then the log of the positive eigenvector of $\bF$ can be approximated in polynomial time.
\end{theorem}
To prove this theorem, we need several lemmas.
\begin{lemma}\label{lem-nu-one}
    Assume that $\bF: \R_{++}^{n}\to \R_{++}^{n}$ is a strongly irreducible homogeneous polynomial map with nonnegative coefficients. 
 Then $\nu(\cA(\bF))\ge 1$.  
\end{lemma}
\begin{proof}
Let
 $\|\y\|_H=\max_{i,j\in[n]} y_i-y_j$ for $\y\in\R^{n}$,  and $\E_n=\{\y\in\R^{n}: y_{n}=0\}$.
Observe $\|\y\|_{\infty}\le \|\y\|_H$ for $\y\in\E_n$.
Assume that $\y\in\E_n$ and $\|\y\|_\infty=1$.   
Observe  
$$\nu(\cA(\bF))=\min_{\y\in\E_n, \|\y\|_\infty=1} \max_{\bb\in\cA_i,i\in[n]}(\bb-(d-1)\be_i)^\top\y=\max_{\bb\in\cA_i,i\in[n]}(\bb-(d-1)\be_i)^\top\y^\star.$$
Let $\|\y^\star\|_H=y_p^\star-y_q^\star\ge 1$. Then $y_p^\star\ge y^\star_i\ge y_q^\star$ for $i\in[n]$.
Owing to the strong irreducibility assumption, there exists $\bb\in \mathcal{A}_i$ such that $\bb_p>0$. 
Then, recalling that $\bF$ is homogeneous of degree $d-1$,
so that $\sum_j \bb_j =d-1$, we get
$$(\bb-(d-1)\be_q)^\top\y^\star
= (\bb-\sum_j b_j \be_q)^\top\y^\star=
\bb_p (y_p^\star-y_q^\star)+
\sum_{j\neq p} \bb_j (y_j^\star -y_q^\star)
\ge y_p^\star-y_q^\star\ge 1,$$
and so $\nu(\cA(\bF))\ge 1$.
  \end{proof}
\begin{lemma}\label{minmxid}  Let $\bb=(b_1,\ldots,b_{n-1})^\top>0$ and denote $b_{\min}=\min(b_1,\ldots,b_{n-1})$.  Assume that $n\ge 2$.  Then
\begin{equation}\label{minmxid1} 
\min_{\|\y\|_{\infty}=1}\{\max(y_1,\ldots,y_{n-1}, -\bb^\top \y\}=\frac{b_{\min}}{1-b_{\min}+\sum_{i=1}^{n-1} b_i} .
\end{equation}
\end{lemma}
\begin{proof}  Let
\begin{equation}\label{ystardef}
t^\star=\min_{\|\y\|_{\infty}=1}\{\max(y_1,\ldots,y_{n-1}, -\bb^\top \y\}=
\max(y_1^{\star},\ldots,y_{n-1}^{\star}, -\bb^\top \y^{\star}), \|\y^\star\|_\infty=1.
\end{equation}
We first show that 
$$1> s:=\frac{b_{\min}}{1-b_{\min}+\sum_{i=1}^{n-1} b_i}\ge t^\star.$$
As $n\ge 2$ the first inequality is straightforward.   The second  inequality follows from
the following  choice of $\y,\|\y\|_\infty=1$.  Assume that $b_i=b_{\min}$.  Set $y_i=-1$, and 
$y_j=a>0, j\in[n-1]\setminus\{i\}$, where $a=b_{\min}-a\sum_{j\in[n-1]\setminus\{i\}}b_j$. Then $a=s$.  
We claim next that $t^\star>0$.  Suppose to the contrary that $t^\star\le 0$.  Then $\y^\star\le 0$.  As $\|\y^\star\|_\infty=1$ it follows that $t^\star=-\bb^\top \y^\star>0$, contradiction.

Assume that $\|\y^\star\|_{\infty}=1$ satisfies \eqref{ystardef}.  Let $I=\{i\in [n-1],|y_i^\star|=1\}$.  As $t^\star<1$ we deduce that $y_i^\star=-1$ for $i\in I$.   We claim that $t^\star=-\bb^\top \y^\star$.  Suppose to the contrary that $t^\star> -\bb^\top \y^\star$.  Assume that $J=\{j\in[n-1],  y_j^\star=t^\star\}$.  By Define $\y(\varepsilon)=(y_1(\varepsilon),\ldots,y_n(\varepsilon))^\top$ as follows
\begin{equation*}
y_i(\varepsilon)=\begin{cases}
y_i^\star \textrm{ if } i\not\in J\\
y_i^\star-\varepsilon \textrm{ if } i\in J.
\end{cases} 
\end{equation*}
The for small enough positive $\varepsilon$ we obtain that 
$$t^\star> \max(y_1(\varepsilon),\ldots,y_{n-1}(\varepsilon),-\bb^\top \y(\varepsilon)),$$ 
contradiction.  We claim that $J=[n-1]\setminus\{I\}$.   Suppose to the contrary that  $k\in [n-1]\setminus(I\cup J)$.  Let $\y\in\R^{n-1}$ obtained from $\y^\star$ by changing the coordinate $y_k^\star (<t^\star)$ to $z_k=t^\star$.  Then $\y$ satisfies the last part of \eqref{ystardef}.   Clearly, $t^\star>-\bb^\top \y$, contradiction. 
Hence$|I|=1$ and $b_i=b_{\min}$ for $i\in I$.
\end{proof}
\begin{lemma}\label{stocestim} Let $2\le n\in\N$, and assume that $A=[a_{ij}]\in\Rp^{n\times n}$ is a stochastic matrix with positive off-diagonal elements.  Denote by $C\in\Rp^{(n-1)\times (n-1)}$ the matrix obtained from $A$ by deleting the row and column $n$.  
\begin{enumerate}[label=(\alph*)]
\item Let 
\begin{equation}\label{defalphbet}
\alpha(A)=\min_{i\in[n-1]}a_{in}, \quad \beta(A)= \min_{i\in[n-1]}a_{ni}.
\end{equation}
Then $\alpha(A),\beta(A)>0$, and
\begin{equation}\label{rhoAest}
\begin{aligned}
\|C\|_\infty= 1-\alpha(A), \qquad
\|(I-C)^{-1}\|_\infty\le \frac{1}{\alpha(A)}.
\end{aligned}
\end{equation}
\item 
 Denote  
\begin{equation*}
\begin{aligned}
\bb^\top=(b_{1},\ldots,b_{n-1}):=(a_{n1},\ldots,a_{n(n-1)})(I-C)^{-1}.
\end{aligned}
\end{equation*}
Then
\begin{equation}\label{estbbi}
\bb>\beta(A)\1, \quad \bb^\top\1\le \frac{1}{\alpha(A)}.
\end{equation}
\item  Let 
$$\E_n=\{\w=(w_1,\ldots,w_{n})^\top\in\R^{n}, w_{n}=0\}.$$
Then the following implication holds: 
$$(A-I)\w\le \0,\w\in\E_n\Rightarrow \w=\0.$$
Denote
\begin{equation}\label{defnuAi}
\omega(A)=\min_{\|\w\|_\infty=1,\w\in\E_n} \{\max(z_1,\ldots,z_{n}), \z=(A-I)\w)\}.
\end{equation}
Then
\begin{equation}\label{nuAiineq}
\begin{aligned}
\omega(A)> \frac{\alpha^2(A)\beta(A)}{\alpha(A) +1}.
\end{aligned}
\end{equation}
\end{enumerate}
\end{lemma}
\begin{proof}
(a)  Clearly, $\alpha(A),\beta(A)>0$.  We have
\begin{equation*}
\begin{aligned}
  \|C\|_\infty:=\max\{\|C\y\|_\infty, \|\y\|_\infty=1\}
  =\max_{i\in[n-1]}\sum_{j=1}^{n-1} a_{ij}=1-\alpha(A).
 \end{aligned}
 \end{equation*}
 As $\|C\|_\infty<1$,  one has the following relations:
 \begin{equation*}
 \|(I-C)^{-1}\|_\infty=\|\sum_{k=0}^\infty C^k\|_\infty\le \sum_{k=0}^\infty \|C^k\|_\infty\le \sum_{k=0}^\infty (1-\alpha(A))^k=\frac{1}{\alpha(A)}.
 \end{equation*}
 
\noindent
(b)  Clearly,  $(I-C)^{-1}>I$.  Hence 
$$\bb>(a_{n1},\ldots,a_{n,n-1})^\top>\beta(A)\1.$$
This shows the first inequality in \eqref{estbbi}.
We now show the second inequality in \eqref{estbbi}.   Observe
\begin{equation*}
\begin{aligned}
\sum_{i=1}^{n-1} b_i=(a_{n1},\ldots,a_{n,n-1})(I-C)^{-1}\1\le 
(a_{n1},\ldots,a_{n(n-1)})(\frac{1}{\alpha(A)}\1)\le \frac{1}{\alpha(A)}.
\end{aligned}
\end{equation*}

\noindent
(c) Let $\bs$ be the positive left eigenvector of $A$: $\bs^\top A=\bs^\top$. Assume to the contrary that $\w\in\E_n\setminus\{ \0\}$ and $(A-I)\w\le\0$.  Observe that $\bs^\top\left((A-I)\w\right)=\0$.  Hence $(A-I)\w=\0$.  So $\w$ is an eigenvector of $A$ corresponding to the eigenvalue $1$.  As $A$ is irreducible, $\w=t\1$ which is impossible as $\w\in \E_n\setminus\{\0\}$.  

We now show \eqref{nuAiineq}.
Let 
\begin{equation*}
\z=(A-I)\w,  \w^\top=(\bv^\top,0),  \z^\top=(\y^\top, z_{n}), \bv,\y\in\R^{n-1}.
\end{equation*}
Then
\begin{equation*}
\begin{aligned}
\y=(C-I)\bv, \quad \bv=-(I-C)^{-1}\y, \\ z_{n+1}=-\bb^\top \y,\quad
\z^\top=(\y^\top,-\bb^\top \y),\quad
\|\w\|_\infty=\|\bv\|_\infty,\\
\omega(A)=\min_{\|\bv\|_\infty=1}\max(y_1,\ldots,y_{n-1},-\bb^\top \y).
\end{aligned}
\end{equation*}
Observe that
\begin{equation*}
  \|\bv\|_\infty=\|(I-C)^{-1}\y\|_\infty\le \|(I-C)^{-1}\|_\infty \|\y\|_\infty\le \frac{\|\y\|_\infty}{\alpha(A)}
\end{equation*}
and so
\begin{equation*}
\omega(A)\ge \alpha(\alpha) \min_{\|y\|_\infty=1}\max(y_1,\ldots,y_{n-1},-\bb^\top \y).
\end{equation*}
Use \eqref{minmxid1} and \eqref{estbbi} to deduce  \eqref{nuAiineq}.
\end{proof}

Finally, a key idea of the proof of \Cref{sirFthm1} is the following
lemma, which controls
the approximation error of the logarithm of the eigenvector
in terms of the optimality gap of the optimization problem consisting
in minimizing the Collatz-Wielandt function.
\begin{lemma}\label{lemma-renormalize}
  Let $\bu^\star$ denote the unique positive eigenvector of $\bF$ with last
  coordinate equal to $1$. Let $\by^\star:= \log \bu^\star = (\bv^\star,0)$,
  and
  \[
A(\by^\star)=[a_{ij}(\by^\star)],  \quad
a_{ij}(\by^\star)
=
\frac{1}{(d-1)F_i(\bu^\star)} \frac{\partial F_i(\bu^\star)}{\partial u_j}
\bu_j^\star, i,j\in[n] \enspace .
  \]
  Then, $A(\by^\star)$ is a stochastic matrix, and for all $\bv\in \R^{n-1}$,
  \begin{align}
    f((\bv,0))-f((\bv^\star,0))\geq (d-1) \omega(A(\by^\star))\|\bv-\bv^\star\|_\infty \enspace.
    \label{e-def-lb}
  \end{align}
\end{lemma}
\begin{proof}
We have that $\by^\star=\log \bu^\star$
  where $\bu^\star$ is an eigenvector of $\bF$, and this entails that $f_i(\by^\star)=f(\by^\star)$
  holds for all $i\in [n]$.
  Use the convexity of each map $f_i$ to deduce that for all $\by\in\R^n$
  and $1\leq i\leq n$,
\begin{equation}\label{convin1}
\begin{aligned}
  f_1(\y)\geq f_i(\log \bu^\star) +B_i(\y-\log\bu^\star) = f(\log \bu^\star)
  +B_i(\y-\log\bu^\star),\text{ where }
  B_i=D f_i (\log \bu^\star),
\end{aligned}
\end{equation}
so that
\[
f(\by)-f(\log \bu^\star) \geq \max_{i\in [n]}B_i(\y-\log \bu^\star) \enspace .
\]
We have $\frac{1}{d-1}B = A(\log u^\star)-I$.
Euler's identity yield that $A(\log\bu^\star)$ is a stochastic matrix.
Then, by definition of $\omega$,
\[
f(\by)-f(\log \bu^\star) \geq (d+1)\omega(A(\log u^\star)) \|\by-\log \bu^\star\|_\infty
\]
from which the lemma follows, taking $\by = (\bv,0)$.
\end{proof}

\begin{proof}[Proof of~\Cref{sirFthm1}]
  By~\Cref{boundKmin} and~\Cref{lem-nu-one}, the unique minimizer $\by^\star\in \E_n$of $f$
  satisfies
  \[
  \|  \by^\star\|_\infty \leq \frac{f(\0)-t_{\min}(f)}{\nu(\mathcal{A}(\bF))}
    \leq R_0:= f(\0)-t_{\min}(f) \enspace .
    \]
    Set
\begin{equation}\label{defabF1}
\alpha(\bF)=\min_{i\in[n-1],\|\y\|_\infty\le R_0}a_{in}(\y), \quad  \beta(\bF)= \min_{i\in[n-1],\|\y\|_\infty\le R_0}a_{ni}(\y),
\end{equation}
and
\[ \gamma(\bF):=\frac{\alpha(\bF)+1}{(d-1)\alpha^2(\bF)\beta(\bF)}
\]
so that, by~\eqref{e-def-lb} and~\eqref{nuAiineq}, for all $\bv\in \R^{n-1}$,
\begin{align}
\|\bv-\bv^\star\| \leq \gamma(\bF) (f((\bv,0))-f((\bv^\star,0))) \enspace. \label{e-cerror}
\end{align}
As $A(\y)$ is stochastic, $\alpha(\bF)\le 1$.
For $\|\x\|_\infty\le R_0$ and $\by=(\bx,0)$, we have the following simple inequalities 
\begin{equation*}
\begin{aligned}
&F_i(e^{\by})\le F_i(\1)e^{(d-1)R_0)}\le \max_{i\in[n]}\big(\sum_{\bb\in \cB_i} f_{\bb,i}\big) e^{(d-1)R_0}=e^{f(\0)}e^{(d-1)R_0},\\
&\frac{\partial F_i(e^{\y})}{\partial u_j}e^{y_j}\ge e^{t_{\min}(f)}e^{-(d-1)R_0},\quad e^{t_{\min}(f)}=\min_{\bb\in\cB_i, i\in[n]} f_{\bb,i}, \text{ and so}\\
&\alpha(\bF),\beta(\bF)\ge \frac{1}{d-1} e^{t_{\min}(f)-f(\0))-2(d-1)R_0}\text{ which entails}\\
&\log\gamma(\bF)\le \log 2(d-1)^2 +(6d-3)( f(\0)- t_{\min}(f)).
\end{aligned}
\end{equation*}
It follows from~\Cref{lemma-renormalize} that to compute an approximation of $\bv^\star$ with a precision of $\delta$ in the sup-norm, it suffices to compute
an approximate minimizer $\bx^\star$ of $f$, in such a way that
$f(\bx^\star)-f_{\min}\leq \varepsilon:= \delta \gamma(\bF)^{-1}$.
Since $\log \gamma(\bF)$ is polynomial in the input size, it follows
from~\Cref{mainthm} that such a minimizer can be computed in a time
which is polynomial in the input size and $|\log\delta|$.
\end{proof}
\section{Polynomial time approximation of $\mu_d(g)$ and applications}\label{sec:symten}
Recall that the $p$-sphere in $\R^{n}$ is defined by
\begin{equation}\label{defSpn}
\rS_p^n=\{\z\in \R^{n}, \|\z\|_p=1\}, \quad p\in[1,\infty].
\end{equation}
Let $\cM_{d,n}\subset \Z_+^{n}$ be the set of all vectors $\ba=(a_1,\ldots,a_{n})^\top$ with nonnegative integer coordinates whose sum is $d$.  Recall that $|\cM_{d,n}|={n+d-1\choose d}$. 
To each $\ba\in\cM_{d,n}$ we correspond the monomials $\z^{\ba}=z_1^{a_1}\cdots z_{n}^{a_{n}}$ of degree $d$.
Let $g(\z)$ be a nonzero homogeneous polynomial of degree $d$ in $n$ variables
with nonnegative coefficients:
\begin{equation}\label{gexpr}
g(\z)=\sum_{\ba\in\cM(d,n)} g_{\ba}\z^{\ba} \enspace.
\end{equation}
A point $\w\in\rB_p(\0,1)$ is called a local maximum of $g$ if there exists an open set  $O\subset \R^{n}$ such that $\w\in O$, and for $\z\in O\cap\rB_p(\0,1)$ one has the inequality  $g(\z)\le g(\w)$.   Assume that $\w\in \rB_p(\0,1)\cap \Rp^{n}$ is a local maximum of $g$.  Clearly, $\|\w\|_p=1$ and $g(\w)>0$.
\subsection{The maximum $\mu_p(g)$}\label{subsec:mupg}
\begin{lemma}\label{critlemma}
Assume that $p\in[1,\infty)$ and $\w=(w_1,\ldots,w_{n})^\top\ge \0, \|\w\|_p=1$ is a local maximum point of $g(\z)$ for $\|\z\|_p\le 1$.  
\begin{enumerate}[label=(\alph*)]
\item Suppose that $w_j>0$.  Then 
\begin{equation*}
\frac{1}{d}\frac{\partial g}{\partial z_j}(\w)=g(\w)w_{j}^{p-1}, \quad j\in[n].
\end{equation*}
\item Assume that $p>1$.  Then \eqref{geignvalvecintro} holds.
\item Assume that $p=1$ and $\w_j=0$.  Then 
\begin{equation*}
-g(\w)\le \frac{1}{d}\frac{\partial g}{\partial z_j}(\w)\le g(\w).
\end{equation*}
\end{enumerate}
\end{lemma}
\begin{proof} Consider the $0$-homogeneous function $h(\z)=g(\z)\|\z\|_p^{-d}$ in $\R^{n}\setminus\{\0\}$.  Then $g(\w)=h(\w)$ is a local maximum of $h(\z)$ on $\R^n\setminus\{\0\}$.
Suppose that $w_j>0$.  Hence $\frac{\partial h}{\partial z_j}(\w)=0$.  This proves \eqref{geignvalvecintro} for $i=j$ in part (a).

Assume now that $w_j=0$.   Consider the one variable function $\phi(t)=h(\w+t\be_j)$.   Note that $\phi(t)$ is differentiable on $\R\setminus\{0\}$. As $\phi(0)$ is the maximum value of $\phi(t)$ it follows that $\phi'(0+)\le 0\le \phi'(0-)$.  

Assume that $p>1$.   We deduce that $\frac{\partial g}{\partial z_j}(\w)=0$.   
Combine this result with part (a) to deduce part (b).

Assume that $p=1$.  Then
\begin{equation*}
\phi'(0+)=\frac{\partial g}{\partial z_j}(\w)- dg(\w), \quad \phi'(0-)=\frac{\partial g}{\partial z_j}(\w)+dg(\w),
\end{equation*}
which implies part (c).
\end{proof}

We now recall the classical characterization of Motzkin-Straus \cite{MS65}, which gives an example of a quadratic form $g$,  where the computation of $\mu_1(g)$ is 
an an NP-complete problem. 
Let $A=[a_{ij}]\in\Rp^{(n)\times (n)}$ be the adjacency matrix of a graph $G$ on $n$ vertices.  Denote by $\omega(G)$ the clique number
of $G$, whose value is an NP-complete problem.   Let $g(\z)=\sum_{i=j=1}^{n} a_{ij}z_iz_j$.
Then
\begin{equation}\label{MSeqbody}
\mu_1(\z^\top A\z)=1-\frac{1}{\omega(G)}.
\end{equation}
The following immediate consequence of this property gives examples of maps $g$ for which approximating $\mu_p(g)$ is an NP-hard problem:
\begin{proposition}\label{MSprop}  Let $g_{G}(\z)=\z^\top A\z$, where $A\in\Rp^{n\times n}$ is an adjacency matrix of a simple graph $G$.   Denote: $g_{p,G}(\z)=\sum_{i=j=1}^{n} a_{ij}z_i^pz_j^p$. 
  Then, it is NP-hard to approximate $\mu_{p}(g_{p,G}^m)$,
  for each pair $(m,p)$ of positive integers.
\end{proposition}
\begin{proof}  Recall that $\mu_1(g_G)=1-1/\omega(G)$.    Clearly,  for $m,p\in\N$ 
\begin{equation*}
\mu_{p}(g_{p,G})=\mu_{1}(g_G), \quad \mu_{p}(g_{p,G}^m)=\mu_1(g_G)^m.
\end{equation*}
The conclusion follows from the NP-hard character of computing the clique number.
\end{proof}
%
%
%
\subsection{Symmetric tensors and homogeneous polynomials}\label{subsec:stenhp}
A partially symmetric tensor $\cG=[g_{i_1,\ldots,i_d}]\in\R_{ps}^{(n+1)^{\times d}}$ is called symmetric if $g_{i_1,\ldots,i_d}$ is invariant under the permutation of indices:
$g_{i_{\sigma(1)},\ldots i_{\sigma(d)}}=g_{i_1,\ldots,i_d}$ for each bijection $\sigma:[d]\to [d]$.  We denote by $\rS^d\R^{n+1}$ the subset of symmetric tensors in $\R_{ps}^{(n+1)^{\times d}}$.   For $\z=(z_1,\ldots,z_{n+1})^\top\in\R^{n+1}$ denote by $\z^{\otimes d}$ the symmetric tensor, rank-one tensor for $\z\ne \0$, whose $(i_1,\ldots,i_d)$ entry is $z_{i_1}\cdots z_{i_d}$.   Assume that $\cG=[g_{i_1,\ldots,i_d}],\cH=[h_{i_1,\ldots,i_d}]\in\rS^d\R^{n+1}$.  The inner product on symmetric tensors is defined by
\begin{equation*}
\langle \cG,\cH\rangle=\sum_{i_1,\ldots,i_d\in[n+1]} g_{i_1,\ldots,i_d}h_{i_1,\ldots,i_d}.
\end{equation*}
Observe that $g(\z)=\langle \cG,\otimes^ d\z \rangle$ is a homogeneous polynomial of degree $d$, and
\begin{equation}\label{nablagform}
\begin{aligned}
\frac{1}{d} \nabla g(\z)=\bF=(G_1(\z),\ldots,G_{n+1}(\z))^\top,\\
G_i(\z)=\sum_{i_2,\ldots,i_d\in[n+1]} g_{i,i_2,\ldots,i_d}z_{i_2}\cdots z_{i_d}. 
\end{aligned}
\end{equation}
Banach's theorem \cite{Ban38} yields that the spectral norm of $\cG\in\rS^d\R^{n+1}$ is given by
\cite{FW20}:
\begin{equation*}
\|\cG\|_{\sigma}:=\max_{\z\in\R^{n+1},\|\z\|_2=1} |\langle \cG,\otimes ^d\z \rangle|.
\end{equation*}
Assume that $\cG$ is a nonnegative tensor, i.e. , $\cG\in\rS^d\Rp^{n+1}$ . Then
\begin{equation}\label{specnormdef}
\|\cG\|_{\sigma}:=\max_{\z\in\R^{n+1},\|\z\|_2=1, \z\ge \0} \langle \cG,\otimes^d\z \rangle=\mu_2(g), \quad \cG\in\rS^d\Rp^n.
\end{equation}
Proposition \ref{MSprop}  yields that for $d=4m, m\in\N$,  the computation of $\|\cG\|_{\sigma}$, for
$\cG\in\rS^d\R^{n+1}_+$ with rational coefficients, is NP-hard.

One of the main results of this paper, which will be proven in subsection \ref{subsec:pthmmupg} , is:
\begin{theorem}\label{poltcompmudg}  Let $g$ be a nonnegative $d$-form with rational coefficients for an integer $d\ge 2$.  Then an $\varepsilon$-approximation of $\mu_q(g)$ for $q=d$ or $q\geq d+1$ can be computed in polynomial time in $\varepsilon$ and the coefficients of $g$.
\end{theorem}
It is of interest to bound $\mu_q(g)$ in terms of $\mu_d(g)$ for $q\in[1,d]$. 
\begin{proposition}\label{mu2mudbds}  Assume that $d\ge 2$ is an integer, and $q\in[1,d]$.  Let $g$ be a nonnegative $d$-form.  Then
\begin{equation}\label{mu2mudbds1}
\mu_q(g)\le \mu_d(g)\le (n+1)^{d-q}\mu_q(g).
\end{equation}
\end{proposition}
\begin{proof}  It is well known that for $\z\in\R^{n+1}$ and $r\ge q$ one has the inequalities
$\|\z\|_r\le \|\z\|_q\le (n+1)^{1-q/r}\|\z\|_r$.  Let $r=d$.  The inequality $\|\z\|_d\le \|\z\|_q$ yields the first inequality in \eqref{mu2mudbds1}.  Assume that  $\mu_d(g)=g(\w)$ for $\w\ge 0, \|\w\|_d=1$.  Let $\bu=\frac{1}{\|\w\|_q}\w$.  Then
\begin{equation*}
\mu_q(g)\ge g(\bu)=\|\w\|_q^{-d}g(\w)\ge (n+1)^{-d(1-q/d)} \mu_d(g),
\end{equation*}
which yields the second inequality in \eqref{mu2mudbds1}. 
\end{proof}

\subsection{Proof of Theorem \ref{poltcompmudg} for $q=d$}\label{subsec:pthmmupg}
Clearly, it is enough to consider the case where $g$ is a nonzero polynomial, and 
 is not a polynomial in a subset of $[n+1]$ variables.  That is $\frac{\partial g}{\partial z_i}>0$ for $\z>\0$.  Next we assume that $\bG(\z)=\frac{1}{d}\nabla \phi(\z)$ is a weakly irreducible polynomial map.   Assume that $\w\in\rS_d^n\cap\Rp^{n+1}$ is a maximum point of $g$ on $\rB_d(\0,1)$.
Lemma \ref{critlemma} claims that 
\begin{equation*}
\bG(\w)=g(\w)\w^{\circ(d-1)}.
\end{equation*}
Hence, $g(\w)$ is an eigenvalue of $\bG$.  As $\bG$ is weakly irreducible then the maximum eigenvalue of $\bG$ is the spectral radius $\rho(\bG)\ge g(\w)$, with the corresponding eigenvector $\bu>\0$ such that
\begin{equation*}
\bG(\bu)=\rho(\bG)\bu^{\circ(d-1)}
\end{equation*}
Without loss of generality we can assume that $\|\bu\|_d=1$.   Let $\bz\circ\by=(z_1y_1,\ldots,z_{n+1}y_{n+1}$ be a coordinatewise product in $\R^{n+1}$.
Then $\bu\circ G(\bu)=\rho(\bG)\bu^{\circ d}$.  Recall the Euler identity $\sum_{i=1}^{n+1} u_iG_i(\bu)=g(\bu)$.  Hence $g(\bu)=\rho(\bG)\le g(\w)$.  Therefore, $g(\bu)=\mu_d(g)$.   Use Theorem \ref{Fcoercthm1} to deduce Theorem  \ref{poltcompmudg}.

Assume now that $g$ is not irreducible.  Consider the simple graph $G$ on $[n+1]$ vertices such that the undirected edge $\{i,j\}$, for $i\ne j$, is in $G$ if $\frac{\partial^2 g}{\partial z_i\partial z_j}>0$ for $\z>\0$.  Then $G$ is a finite union of connected components $G=\cup_j^m G_j$, where $2\le m\le n+1$.  That is $\z=(\z_1^\top,\ldots,\z^\top_m)^\top$ where $\z_j\in\R^{n_j}$ and 
$g(\z)=\sum_{j=1}^m g_j(\z_j)$, where each $g_j$ is a nonzero nonnegative weakly irreducible polynomial.   Hence we can compute with $\varepsilon$ precision each $\mu(g_j)$.

Next observe that
\begin{equation*}
g(\z)\le \sum_{i=1}^m \mu_d(g_j)\|\z_j\|_d^d\le \big(\max_{j\in [m]}\mu_d(g_j)\big)\|\z\|_d^d.
\end{equation*}
Hence, 
\begin{equation}\label{charmudgred}
\mu_d(g)=\max_{j\in[m]}\mu_d(g_j).
\end{equation}
This proves Theorem \ref{poltcompmudg} in this case.
\subsection{Bounding the clique number of uniform hypergraphs}\label{subsec:unhgr}
Let $2^{[n]}_d$ denote all subsets of $[n]$ of cardinality $d\in[n]$.
Assume that $\cE\subset 2^{[n]}_d$.  A simple $d$-uniform hypergraph corresponding to the set of hyperedges $\cE$, denoted as $\cH=([n],\cE)$  can be described by a symmetric adjacency tensor $\cH(\cE)=[h_{i_1,\ldots,i_d}]\in \rS^d\R^{n}$, where
\begin{equation*}
h_{i_1,\ldots,i_d}=\begin{cases}
1 \textrm{ if } \{i_1,\ldots,i_d\}\in\cE,\\
0 \textrm{ otherwise}.
\end{cases}
\end{equation*}
More generally, equipping the hyperedge $\{i_1,\ldots,i_d\}$ with a weight
$f_{i_1,\ldots,i_d}$, we get a symmetric tensor
\begin{equation*}
\cF(\cE)=[f_{i_1,\ldots,i_d}]\in \rS^d\R^{n},  \quad f_{i_1,\ldots,i_d}>0\iff \{i_1,\ldots,i_d\}\in\cE.
\end{equation*}


We call $\rho(\cF(\cE))$ the spectral radius of the weighted hypergraph.
Theorem \ref{poltcompmudg} shows that $\rho(\cF(\cE))=\mu_d(\cF(\cE)\times\otimes^d \z)$ can be approximated in polynomial time. 

We now use this result to provide a bound for the clique number of a hypergraph
that can be obtained in polynomial time.
Recall that a {\em clique} $\cC$ in $\cE$ is subset of $[n]$ such that $2^{\cC}_d\subset \cE$.
Denote by $\omega(\cE)$ the maximum cardinality of a clique in $\cE$.
\begin{proposition}\label{ubcliqueghr} Let $\cE\subset 2^{[n]}_d$.  Denote by $\cH(\cE)\in\rS^d\R^{n}$ the symmetric tensor induced by $\cE$.  Then,
\begin{equation}\label{ubcliqueghr1} 
\prod_{i=1}^{d-1} (\omega(\cE)-i)\le \rho(\cH(\cE)).
\end{equation}
This inequality is sharp if $\cE$ is a union of disjoint cliques in $2^{[n]}_d$.
\end{proposition}
\begin{proof} Set $g(\z)=\cH(\cE)\times \otimes^d\z$.  Then $\rho(\cH(\cE))=\mu_d(g)$.  Assume that $\cC\subset [n]$ is a clique in $\cE$.  Let 
\begin{equation*}
\z=(z_1,\ldots,z_{n})^\top=\begin{cases}
z_i=\frac{1}{|\cC|^{1/d}} \textrm{ if }i\in\cC,\\
z_i=0 \quad\quad\quad \textrm{ if } i\not\in \cC.
\end{cases}
\end{equation*}
Then 
\begin{equation*}
g(\z)=\frac{d!{|\cC|\choose d}}{|\cC|}=\prod_{i=1}^{d-1} (|\cC|-i)\le \mu_d(g)=\rho(\cH(\cE)).
\end{equation*}
Choose $\cC$ to satisfy $|\cC|=\omega(\cE)$ to deduce \eqref{ubcliqueghr1}.

Assume that $\cE=\cup_{j=1}^k 2^{\cC_j}_d$, where $\cup_{j=1}^k \cC_j\subset [n]$, and $\cC_i\cap\cC_j=\emptyset$ for $i\ne j$.   Then $g=\sum_{j=1}^k g_j(\z_j)$ is a decomposition of $g$ to its irreducible components.  The proof of Theorem \ref{poltcompmudg} shows that it is enough to show the sharpness of \eqref{ubcliqueghr1} when $\cE=2^{\cC}_d$.   Furthermore we can assume that $\cC=[n]$.  Then 
\begin{equation*}
\frac{g(\x)}{d!}=\sigma_d(\z)=:\sum_{1\le i_1<\cdots<i_d\le n} z_{i_1}\cdots z_{i_d}.
\end{equation*}
Use Maclaurin's and H\"older's inequalities to deduce
\begin{equation*}
\sigma_d(\z)\le {n \choose d}\big(\|\z\|_1/(n)\big)^d\le  {n \choose d}(n)^{-1}\|\z\|_d^d.
\end{equation*}
Assume that $\|\z\|_d=1$ to deduce the equality for $\cE=2^{[n]}_d$.
\end{proof}
We remark that for $d=2$, the inequality $\omega(G)\le 1+\rho(G)$ is follows from a stronger inequality \cite{Wil67}
\begin{equation*}
\chi(G)\le 1+\rho(G),
\end{equation*}
where $\chi(G)$ is the chromatic number of $G$, as $\chi(G)\ge \omega(G)$.
For $d\ge 3$ the above inequality is closely related to the inequalities in \cite{RBP09,XQ15}.

We now recall some results in \cite{RBP09}.  Let $G=([n],E)$ be a simple graph.  Assume that $\omega(G)\ge d$ for  $2<d\in[n]$.  Denote by 
$E_d\subset 2^{[n]}_d$ the set of all $d$ - vertices $\{i_1,\ldots,i_d\}$ that form a clique in $G$.   Clearly, $\omega(E_d)=\omega(G)$.  Apply the inequality \eqref{ubcliqueghr1}  for $\cE=E_d$.   Use this inequality to find an upper bound on $\omega(G)$.  
It is shown in  \cite{RBP09} that in many cases this inequality is better than $\omega(G)\le 1+\rho(G)$.  More precisely,  it is conjectured in \cite{RBP09} that  the upper for $\omega(G)$ using \eqref{ubcliqueghr1} is decreasing with $d$.
\subsection{Dihypergraphs}\label{subsec:dihyper}
In this subsection we consider the following types of weighted dihypergraph with the set of vertices $[n]$.  Denote by $\Phi_{d-1}^{[n]}$  the set all multi-subsets of cardinality $d-1$  of $[n]$.  Each multiset $\phi=\{\phi_1,\ldots,\phi_{d-1}\}\in [n]^{d-1}$, is associated with a unique monomial of degree $d-1$ in $n$ variables: $\x^{\ba}=\prod_{i=1}^{d-1} x_{\phi_i}$.    For each $j\in[n]$, let $\Phi_j\subseteq  \Phi_{d-1}^{[n]}$.   Then a dihypergraph $\overrightarrow{\cH}=([n], \cup_{j=1}^{n}\Phi_j)$ has dihyperedge $\overrightarrow{j\phi}$ for all $\phi\in\Phi_j, j\in[n]$.
 That is,  the tensor $\cH=[h_{i_1,\ldots,i_d}]\in  \R_{ps,+}^{(n)^{\times d}}$ represents $\overrightarrow{\cH}$ if and only if 
 \begin{equation*}
h_{i_1,\ldots,i_d}=\begin{cases}
1 \textrm{ if } \{i_2,\ldots,i_d\}\in\Phi_{i_1},\\
0 \textrm{ otherwise}.
\end{cases}
\end{equation*}

Similarly, a nonnegative weighted dihypergraph is represented by unique $\cT\in \R_{ps,+}^{(n)^{\times d}}$.   
Let $\bF:\R^{n}\to\R^{n}$ be the homogeneous maps induced by $\cT$.  
Then $\rho(\bF)$ is $\rho(\overrightarrow{\cH})$-the spectral radius of $\overrightarrow{\cH}$.
The dihypergraph $\overrightarrow{\cH}$ is weakly irreducible if $\bF$ is weakly irreducible.  Our results show that if $\overrightarrow{\cH}$ is weakly irreducible then
it spectral radius is polynomial time approximable.

\subsection{Extension to posynomial maps}
\label{subsec:quashommapspol}
Posynomials are similar to polynomials with nonnegative coefficients but the exponents of posynomials are allowed to take real values
(or rational values when complexity issues are considered).

We now explain how~\Cref{Fcoercthmnew} and \Cref{sirFthm} carry
over to the case of {\em posynomial} maps.
We will show in addition that for a rational $p\ge d+1$, and a posynomial
form of degree $d\ge 2$, one can approximate $\mu_p(g)$ in polynomial time.

More precisely, for $\z\in\Rpp^{n}$ and $\ba\in\Rpp^{n}$ let $\z^{\ba}=z_1^{a_1}\cdots z_{n}^{a_{n}}$.  We assume $0^0=1$.  For a real $d\ge 1$ denote  
\begin{equation}\label{defNdn}
\begin{aligned}
&\cN_{d,n}=\{\ba\in \R^{n}_+, \|\ba\|_1=d\}, d>0,\\
&\cN'_{d,n}=\{\ba\in \R^{n}_+, a_i\in\{0\}\cup[1,d], i\in[n],\|\ba\|_1=d\}, d\in(1,\infty).\\
\end{aligned}
\end{equation}
In what follows we discuss a nonnegative homogeneous  posynomial $g_{ \cA}$ on $\R_+^{n}$ of degree $d>1$  given by
$g_{\cA} =\sum_{\ba\in \cA} g_{\ba} \z^{\ba}$, where $g_{\ba}> 0$, and   $\cA$ is a finite subset of $\cN_{d,n}$.  Observe that $g_{\cA}$ is continuous on $\R^{n}$, and $\nabla g_{\cA}$ is continuous on $\R_+^{n}$ if $\cA\subset \cN'_{d,n}$ (the set $\cN'_{d,n}$ excludes
exponents in $(0,1)$ which would make the gradient singular at the origin).
Note that for a given $\cA\subset \cN_{d,n}$ there exists a positive integer $k$ such that $k\cA\subset \cN'_{d,n}$.

A homogeneous posynomial map $\bF:\R^{n}\to \R^{n}$ of degree $d-1> 0$ is given by
\begin{equation}\label{defqhF}
F_i(\z)=\sum_{\bb\in \cB_i} f_{\bb,i} \z^{\bb}, \quad \cB_i \subset \cN_{d-1,n}, i\in[n].
\end{equation}

Assume first that $\cA$ and each $\cB_i$ is a subset of $\cN'_{d-1,n}$.
It is straightforward to show that Theorems  \ref{Fcoercthmnew} and  \ref{sirFthm} apply to homogeneous posynomials.

We now consider the case where $\cA$ and each $\cB_i$ is a subset of $\cN_{d-1,n}$.   For $k\in\N$ denote $g_k(\x)=g(\z^{\circ k}), \bF_k=\bF(\z^{\circ k})$ respectively.  Then $g$ is irreducible iff $g_k$ is irreducible, and $\bF$ is weakly irreducible (strongly irreducible) iff  $\bF_k$ is weakly irreducible (strongly irreducible). 
Furthermore,  $\mu_{p}(g)=\mu_{kp}(g_k)$, and 
$\rho(\bF)=\rho(\bF_k)$.  Furthermore, if $\bF_k$ is weakly irreducible then 
\begin{equation*}
\bF(\x)=\rho(\bF)\x^{\circ (d-1)}\iff \bF_k(\x)=\rho(\bF_k)\x^{\circ k(d-1)}.
\end{equation*}
Hence, without loss of generality we can assume that $\cA$ and $\cB_i$ are subsets of $\cN'_{d,n}$.

\subsection{The polynomial computability of $\mu_p(g)$ for a rational $p> d$}\label{subsec:pcmupg}
Assume that $g(\z)$ a nonzero nonnegative  posynomial of rational degree $d\ge 2$, which is a function of a subset of $z_1,\ldots,z_{n}$ variables.   Set $\tilde\z=(\z,z_{n+1})\in\R^{n+1}$.  Define 
\begin{equation*} 
\tilde g(\tilde\z):=z_{n+1}^{p-d}g(\z), \quad \z,z_{n+1}\ge 0.
\end{equation*}
Clearly, $\tilde g(\tilde \z)$ is homogeneous posynomial of degree  $p$.  Assume first that $p\ge d+1$.  Set $\bF=\frac{1}{p}\nabla \tilde g$.  Observe that in the simple graph $\overset{\rightarrow} G(\bF)$, all vertices are connected to the vertex $n+1$.  Hence $\tilde g$ is weakly irreducible.
Theorem \eqref{poltcompmudg} for $\tilde g$ yields that we can approximate the value of $\mu_p(\tilde g)$ in polynomial time.
It is left to show that this implies that we can find an approximation of $\mu_p(g)$ in polynomial time.   Observe that $\mu_p(\tilde g)$ achieved for $\tilde \z^\star$ where 
$z_{n+1}^\star\in (0,1)$.  Fix $t\in (0,1)$ and consider the maximum of $\tilde g(\tilde \z)$ for $\|\z\|_p^p=(1-t^p)$.   This maximum is $\mu_p(g)t^{p-d}(1-t^p)^{d/p}$.  Set
\begin{equation*}
\alpha_{d,p}=\max_{t\in[0,1]}t^{p-d}(1-t^p)^{d/p}.
\end{equation*} 
This maximum is achieved at:
\begin{equation*}
t^{\star}_{d,p}=\big( 1-\frac{d}{p}\big)^{1/p}.
\end{equation*}
Thus $\mu_p(\tilde g)=\alpha_{d,p} \mu_p(g)$.
Thus we can approximate $\mu_p(g)$. In the case $p\in(d, d+1)$ consider $g^k$, where $k$ is a posiitve integer satisfying $kp\ge kd+1$. Apply now the previous result.

\subsection{Polynomial time computability of some $\|A\|_{p,q}$ norms for nonnegative matrices}\label{sec:pqmnorms}
Let 
\begin{equation*}
\|A\|_{p,q}=\max_{\|\x\|_p\le 1}\|A\x\|_q=\max_{\|\x\|_p\le 1}\left(\sum_{i=1}^m |(A\x)_i|^q\right)^{1/q}, \quad A\in\R^{m\times n},
\end{equation*}
be a mixed $p-q$-matrix norm.
It is known that $\|A\|_{p,q}$ is NP-Hard to compute for $1\le q< p\le \infty$ \cite{Ste05} and for $p=q\ne 1,2,\infty$ \cite{HO10}.  It is conjectured in \cite{Ste05} that there are only three cases in which mixed norms are
computable in polynomial time: First, p = 1, and q is any rational number larger
than or equal to 1.  Second, $q = \infty$  and p is any rational number larger than or equal to 1. Third,  $p = q = 2$.
\begin{theorem}\label{poscomp}  Let $A\in\Q^{m\times n}_{+}$,    Assume that $q\in\N, q\le p\in\Q$.  Then
the mixed norm $\|A\|_{p,q}$ is polynomial time computable.
\end{theorem}
\begin{proof} As $A$ is nonnegative we deduce
\begin{equation*}
\|A\|_{p,q}^q=\max_{\|\x\|_p\le 1} \sum_{i=1}^m |(A\x)_i|^q= \max_{\|\x\|_p\le 1, \x\in\R^n_+} \sum_{i=1}^m (A\x)_i^q.
\end{equation*}
Clearly,  $g(\x)=\sum_{i=1}^m (A\x)_i^{q}$ is a homogeneous polynomial of degree $q$ with nonnegative coefficients.   Assume first that $q=1$.
Then 
\begin{equation*}
g(\x)=\sum_{j=1}^m \left(\sum_{i=1}^n a_{ij}\right)x_j\Rightarrow   \|A\|_{p,1}=\left(\sum_{j=1}^m (\sum_{i=1}^n a_{ij})^r\right)^{1/r}, r=\frac{p}{p-1}.
\end{equation*}
Hence, $\|A\|_{p,1}$ is computable for $1\le p\in \Q$.

Assume that $2\le q\in\N$.
As $q\le p$, the result of  subsection \ref{subsec:pcmupg} yields that the maximum of $g(\x)$ over the unit ball in $\|\cdot\|_p$ is polynomial time computable.
\end{proof}

 \section*{Acknowledgment}
 The work of the first author is partially supported by the Simons Collaboration Grant for Mathematicians.   The work of the second author was partially supported by ANR and by FACCTS.

\bibliographystyle{plain}

\end{document}